\documentclass[11pt]{article}

\usepackage{amsmath}
\usepackage{amssymb}
\usepackage{amsfonts}
\usepackage{mathrsfs}
\usepackage{graphicx}
\usepackage{color,xcolor}
\usepackage{float}
\usepackage{subfig}
\usepackage[normalem]{ulem}
\usepackage[linktocpage,colorlinks,linkcolor=blue,anchorcolor=blue, citecolor=blue,urlcolor=blue]{hyperref}
\usepackage{pdfpages}
\usepackage[inline]{enumitem}
\usepackage{multirow}
\usepackage{makecell}
\usepackage{breakcites}
\usepackage{bbm}
\usepackage{wrapfig}
\usepackage{wrapstuff}
\usepackage[thinc]{esdiff}
\allowdisplaybreaks[4]
\usepackage{rotating}

\newtheorem{theorem}{Theorem}[section]
\newtheorem{definition}[theorem]{Definition}
\newtheorem{lemma}[theorem]{Lemma}
\newtheorem{example}[theorem]{Example}
\newtheorem{corollary}[theorem]{Corollary}
\newtheorem{conjecture}[theorem]{Conjecture}
\newtheorem{proposition}[theorem]{Proposition}

\newtheorem{remark}[theorem]{Remark}

\def\bd{\boldsymbol}
\def\pn{\par\smallskip\noindent}

\DeclareMathAlphabet\mathbfcal{OMS}{cmsy}{b}{n}

\newcommand{\N}{\mathbb{N}}
\newcommand{\R}{\mathbb{R}}

\newcommand{\bbH}{\mathbb{H}}

\newcommand{\F}{\mathcal{F}}

\newcommand{\bbB}{\mathbb{B}}

\newcommand{\bbI}{\mathbb{I}}
\newcommand{\bbJ}{\mathbb{J}}

\newcommand{\bbT}{\mathbb{T}}
\newcommand{\bbX}{\mathbb{X}}
\newcommand{\bbZ}{\mathbb{Z}}

\newcommand{\ba}{\boldsymbol{a}}

\newcommand{\bdd}{\boldsymbol{d}}

\newcommand{\bI}{\boldsymbol{I}}
\newcommand{\be}{\boldsymbol{e}}
\newcommand{\bg}{\boldsymbol{g}}

\newcommand{\bx}{\boldsymbol{x}}
\newcommand{\by}{\boldsymbol{y}}
\newcommand{\bz}{\boldsymbol{z}}
\newcommand{\bu}{\boldsymbol{u}}
\newcommand{\bv}{\boldsymbol{v}}
\newcommand{\bw}{\boldsymbol{w}}

\newcommand{\bP}{\boldsymbol{P}}

\newcommand{\bQ}{\boldsymbol{Q}}

\newcommand{\bs}{\boldsymbol{s}}
\newcommand{\bS}{\boldsymbol{S}}
\newcommand{\bT}{\boldsymbol{T}}

\newcommand{\bX}{\boldsymbol{X}}
\newcommand{\bY}{\boldsymbol{Y}}
\newcommand{\bZ}{\boldsymbol{Z}}
\newcommand{\bbE}{\mathbb{E}}

\newcommand{\bbV}{\mathbb{V}}
\newcommand{\bbM}{\mathbb{M}}
\newcommand{\bbQ}{\mathbb{Q}}

\newcommand{\T}{\textnormal{T}}

\newcommand{\Signuuuu}{\operatorname{Sign}(\bu\bu^{\T}-\bu^{\star}{\bu^{\star}}^{\T})}
\newcommand{\Signuiujuiuj}{\operatorname{Sign}(u_i u_j-u^{\star}_i u^{\star}_j)}

\usepackage{booktabs}
\usepackage{graphicx}
\usepackage{stmaryrd}
\usepackage{mathtools}
\def\multiset#1#2{\ensuremath{\left(\kern-.3em\left(\genfrac{}{}{0pt}{}{#1}{#2}\right)\kern-.3em\right)}}

\usepackage{multicol}
\usepackage{nicematrix}
\usepackage{stackengine}
\usepackage{centernot}

\stackMath

\usepackage{algorithm}
\usepackage{algorithmic}

\usepackage{stackengine}
\stackMath
\newcommand\tsup[2][2]{%
 \def\useanchorwidth{T}%
  \ifnum#1>1%
    \stackon[-.5pt]{\tsup[\numexpr#1-1\relax]{#2}}{\scriptscriptstyle\sim}%
  \else%
    \stackon[.5pt]{#2}{\scriptscriptstyle\sim}%
  \fi%
}

\newcommand\restr[2]{{
  \left.\kern-\nulldelimiterspace 
  #1 
  \littletaller 
  \right|_{#2} 
  }}

\newcommand{\littletaller}{\mathchoice{\vphantom{\big|}}{}{}{}}

\usepackage{tikzit}

\tikzstyle{red dot}=[fill=red, draw=black, shape=circle]
\tikzstyle{green dot}=[fill=green, draw=black, shape=circle]
\tikzstyle{medium box}=[fill=white, draw=black, shape=rectangle, minimum width=0.75cm, minimum height=1cm]
\tikzstyle{0.1pt}=[fill=black, draw=black, shape=circle, minimum size=0.1pt, inner sep=0pt]
\tikzstyle{0.5pt}=[fill=black, draw=black, shape=circle, minimum size=0.5pt, inner sep=0pt]
\tikzstyle{1pt}=[fill=black, draw=black, shape=circle, minimum size=1pt, inner sep=0pt]
\tikzstyle{2pt}=[fill=black, draw=black, shape=circle, minimum size=2pt, inner sep=0pt]
\tikzstyle{3pt}=[fill=black, draw=black, shape=circle, minimum size=3pt, inner sep=0pt]
\tikzstyle{3pthide}=[fill=black, draw=black, shape=circle, minimum size=3pt, inner sep=0pt, opacity=0.1]
\tikzstyle{5pt}=[fill=black, draw=black, shape=circle, minimum size=5pt, inner sep=0pt]
\tikzstyle{big circle}=[fill=none, draw=black, shape=circle, minimum size=20cm, inner sep=0pt, ultra thick]
\tikzstyle{0.5none}=[fill=none, draw=none, shape=circle, scale=0.5]
\tikzstyle{small circle}=[fill=none, draw=black, shape=circle, minimum size=5cm, inner sep=0pt, ultra thick]

\tikzstyle{directional}=[>=stealth, ->]
\tikzstyle{thick direc}=[>=stealth, ->, very thick]
\tikzstyle{dashes}=[-, densely dotted, thick]
\tikzstyle{wavy}=[-, snake it, thick]
\tikzstyle{big dashes}=[-, thick, dashed, dash pattern=on 4mm off 2mm, fill=cyan]
\tikzstyle{blue directional}=[draw=blue, ->, very thick, >=stealth]
\tikzstyle{red directional}=[draw=red, ->, very thick, >=stealth]
\tikzstyle{green directional}=[draw=green, ->, very thick, >=stealth]
\tikzstyle{thin line}=[-, very thin]
\tikzstyle{purple directional}=[->, draw=magenta, very thick, >=stealth]
\tikzstyle{pure shade}=[-, draw=none, fill={rgb,255: red,191; green,191; blue,191}, opacity=0.5]
\tikzstyle{pure rshade}=[-, draw=none, fill={rgb,255: red,255; green,0; blue,0}, opacity=0.5]
\tikzstyle{pure gshade}=[-, draw=none, fill={rgb,255: red,0; green,255; blue,0}, opacity=0.5]
\tikzstyle{blue line}=[-, fill=none, draw=blue, very thick]
\tikzstyle{purple line}=[-, draw=magenta, very thick]
\tikzstyle{line shade}=[-, draw=black, fill={rgb,255: red,191; green,191; blue,191}]
\tikzstyle{line shade 2}=[-, draw=black, fill={rgb,255: red,128; green,128; blue,128}]
\tikzstyle{thick line}=[-, very thick]
\tikzstyle{tlhide}=[-, very thick, opacity=0.1]

\usepackage{tikz}
\usepackage{tikz-3dplot}
\usetikzlibrary{calc}
\usepackage{pgfplots}
\usepackage{xxcolor}
\pgfplotsset{compat=1.16}
\usepgfplotslibrary{fillbetween}
\pgfdeclareradialshading[tikz@ball]{ball}{\pgfqpoint{0bp}{0bp}}{%
 color(0bp)=(tikz@ball!0!white);
 color(7bp)=(tikz@ball!0!white);
 color(15bp)=(tikz@ball!70!black);
 color(20bp)=(black!70);
 color(30bp)=(black!70)}
\makeatother

\tikzset{viewport/.style 2 args={
    x={({cos(-#1)*1cm},{sin(-#1)*sin(#2)*1cm})},
    y={({-sin(-#1)*1cm},{cos(-#1)*sin(#2)*1cm})},
    z={(0,{cos(#2)*1cm})}
}}

\pgfplotsset{only foreground/.style={
    restrict expr to domain={rawx*\CameraX + rawy*\CameraY + rawz*\CameraZ}{-0.05:100},
}}
\pgfplotsset{only background/.style={
    restrict expr to domain={rawx*\CameraX + rawy*\CameraY + rawz*\CameraZ}{-100:0.05}
}}

\def\addFGBGplot[#1]#2;{
    \addplot3[#1,only background, opacity=0.25] #2;
    \addplot3[#1,only foreground] #2;
}

\usetikzlibrary{3d}
\usetikzlibrary{calc}
\usetikzlibrary{babel} 

\pgfmathsetmacro\xx{1/sqrt(2)}
\pgfmathsetmacro\xy{1/sqrt(6)}
\pgfmathsetmacro\zy{sqrt(2/3)}

\usetikzlibrary{3d}
\usetikzlibrary{calc, shadings} 
\usetikzlibrary{positioning,arrows.meta}
\usetikzlibrary{trees}
\usepgfplotslibrary{fillbetween}
\DeclareMathAlphabet\mathbfcal{OMS}{cmsy}{b}{n}

\makeatletter
\def\tikz@lib@cuboid@get#1{\pgfkeysvalueof{/tikz/cuboid/#1}}

\def\tikz@lib@cuboid@setup{%
   \pgfmathsetlengthmacro{\vxx}%
      {\tikz@lib@cuboid@get{xscale}*cos(\tikz@lib@cuboid@get{xangle})*1cm}
   \pgfmathsetlengthmacro{\vxy}%
      {\tikz@lib@cuboid@get{xscale}*sin(\tikz@lib@cuboid@get{xangle})*1cm}
   \pgfmathsetlengthmacro{\vyx}%
      {\tikz@lib@cuboid@get{yscale}*cos(\tikz@lib@cuboid@get{yangle})*1cm}
   \pgfmathsetlengthmacro{\vyy}%
      {\tikz@lib@cuboid@get{yscale}*sin(\tikz@lib@cuboid@get{yangle})*1cm}
   \pgfmathsetlengthmacro{\vzx}%
      {\tikz@lib@cuboid@get{zscale}*cos(\tikz@lib@cuboid@get{zangle})*1cm}
   \pgfmathsetlengthmacro{\vzy}%
      {\tikz@lib@cuboid@get{zscale}*sin(\tikz@lib@cuboid@get{zangle})*1cm}
}

\def\tikz@lib@cuboid@draw#1--#2--#3\pgf@stop{%
    \begin{scope}[join=bevel,x={(\vxx,\vxy)},y={(\vyx,\vyy)},z={(\vzx,\vzy)}]
       \begin{scope}[canvas is yz plane at x=#1]
          \draw[cuboid/all faces,cuboid/edges,cuboid/right face] 
                (0,0) -- ++(#2,0) -- ++(0,-#3) -- ++(-#2,0) -- cycle;
          \draw[cuboid/all grids,cuboid/right grid] (0,0) grid (#2,-#3);
       \end{scope}
       \begin{scope}[canvas is xy plane at z=0]
          \draw[cuboid/all faces,cuboid/edges,cuboid/front face] 
                (0,0) -- ++(#1,0) --  ++(0,#2) -- ++(-#1,0) -- cycle;
          \draw[cuboid/all grids,cuboid/front grid] (0,0) grid (#1,#2);
       \end{scope}
       \begin{scope}[canvas is xz plane at y=#2]
          \draw[cuboid/all faces,cuboid/edges,cuboid/top face] 
                (0,0) -- ++(#1,0) --  ++(0,-#3) -- ++(-#1,0) -- cycle;
          \draw[cuboid/all grids,cuboid/top grid] (0,0) grid (#1,-#3);
       \end{scope}
       \draw[cuboid/hidden edges] (0,#2,-#3) -- (0,0,-#3) -- (0,0,0) 
                (0,0,-#3) -- ++(#1,0,0);
       \begin{scope}[canvas is yz plane at x=#1]
          \draw[cuboid/all faces,cuboid/right face,cuboid/edges,fill opacity=0] 
                (0,0) -- ++(#2,0) -- ++(0,-#3) -- ++(-#2,0) -- cycle;
       \end{scope}
       \begin{scope}[canvas is xy plane at z=0]
          \draw[cuboid/all faces,cuboid/front face,cuboid/edges,fill opacity=0] 
                (0,0) -- ++(#1,0) --  ++(0,#2) -- ++(-#1,0) -- cycle;
       \end{scope}
       \begin{scope}[canvas is xz plane at y=#2]
          \draw[cuboid/all faces,cuboid/top face,cuboid/edges,fill opacity=0] 
                (0,0) -- ++(#1,0) --  ++(0,-#3) -- ++(-#1,0) -- cycle;
       \end{scope}
       \path (0,#2,0) coordinate (-left top front)
                      coordinate (-left front top)
                      coordinate (-top left front)
                      coordinate (-top front left)
                      coordinate (-front top left)
                      coordinate (-front left top);
       \path (0,#2,-#3) coordinate (-left top rear)
                        coordinate (-left rear top)
                        coordinate (-top left rear)
                        coordinate (-top rear left)
                        coordinate (-rear top left)
                        coordinate (-rear left top);
       \path (0,0,-#3) coordinate (-left bottom rear)
                       coordinate (-left rear bottom)
                       coordinate (-bottom left rear)
                       coordinate (-bottom rear left)
                       coordinate (-rear bottom left)
                       coordinate (-rear left bottom);
       \path (0,0,0) coordinate (-left bottom front)
                     coordinate (-left front bottom)
                     coordinate (-bottom left front)
                     coordinate (-bottom front left)
                     coordinate (-front bottom left)
                     coordinate (-front left bottom);
       \path (#1,#2,0) coordinate (-right top front)
                       coordinate (-right front top)
                       coordinate (-top right front)
                       coordinate (-top front right)
                       coordinate (-front top right)
                       coordinate (-front right top);
       \path (#1,#2,-#3) coordinate (-right top rear)
                         coordinate (-right rear top)
                         coordinate (-top right rear)
                         coordinate (-top rear right)
                         coordinate (-rear top right)
                         coordinate (-rear right top);
       \path (#1,0,-#3) coordinate (-right bottom rear)
                        coordinate (-right rear bottom)
                        coordinate (-bottom right rear)
                        coordinate (-bottom rear right)
                        coordinate (-rear bottom right)
                        coordinate (-rear right bottom);
       \path (#1,0,0) coordinate (-right bottom front)
                      coordinate (-right front bottom)
                      coordinate (-bottom right front)
                      coordinate (-bottom front right)
                      coordinate (-front bottom right)
                      coordinate (-front right bottom);
       \coordinate (-left center) at (0,.5*#2,-.5*#3);
       \coordinate (-right center) at (#1,.5*#2,-.5*#3);
       \coordinate (-top center) at (.5*#1,#2,-.5*#3);
       \coordinate (-bottom center) at (.5*#1,0,-.5*#3);
       \coordinate (-front center) at (.5*#1,.5*#2,0);
       \coordinate (-rear center) at (.5*#1,.5*#2,-#3);
       \coordinate (-center) at (.5*#1,.5*#2,-.5*#3);
       \path (0,#2,-.5*#3) coordinate (-left top center) 
                           coordinate (-top left center);
       \path (.5*#1,#2,-#3) coordinate (-top rear center)
                            coordinate (-rear top center);
       \path (#1,#2,-.5*#3) coordinate (-right top center)
                            coordinate (-top right center);
       \path (.5*#1,#2,0) coordinate (-top front center)
                          coordinate (-front top center);
       \path (0,0,-.5*#3) coordinate (-left bottom center) 
                           coordinate (-bottom left center);
       \path (.5*#1,0,-#3) coordinate (-bottom rear center)
                            coordinate (-rear bottom center);
       \path (#1,0,-.5*#3) coordinate (-right bottom center)
                            coordinate (-bottom right center);
       \path (.5*#1,0,0) coordinate (-bottom front center)
                          coordinate (-front bottom center);
       \path (0,.5*#2,0) coordinate (-left front center) 
                           coordinate (-front left center);
       \path (0,.5*#2,-#3) coordinate (-left rear center)
                            coordinate (-rear left center);
       \path (#1,.5*#2,0) coordinate (-right front center)
                            coordinate (-front right center);
       \path (#1,.5*#2,-#3) coordinate (-right rear center)
                          coordinate (-rear right center);
    \end{scope}
}

\tikzset{
  pics/cuboid/.style = {
    setup code = \tikz@lib@cuboid@setup,
    background code = \tikz@lib@cuboid@draw#1\pgf@stop
  },
  pics/cuboid/.default={1--1--1},
  cuboid/.is family,
  cuboid,
  all faces/.style={fill=white},
  all grids/.style={draw=none},
  front face/.style={},
  front grid/.style={},
  right face/.style={},
  right grid/.style={},
  top face/.style={},
  top grid/.style={},
  edges/.style={},
  hidden edges/.style={draw=none},
  xangle/.initial=0,
  yangle/.initial=90,
  zangle/.initial=210,
  xscale/.initial=1,
  yscale/.initial=1,
  zscale/.initial=0.5
}

\newcommand{\tikzcuboidreset}{
\tikzset{cuboid,
  all faces/.style={fill=white},
  all grids/.style={draw=none},
  front face/.style={},
  front grid/.style={},
  right face/.style={},
  right grid/.style={},
  top face/.style={},
  top grid/.style={},
  edges/.style={},
  hidden edges/.style={draw=none},
  xangle=0,
  yangle=90,
  zangle=210,
  xscale=1,
  yscale=1,
  zscale=0.5
}
}

\newcommand{\tikzcuboidset}{\@ifstar\tikzcuboidset@star\tikzcuboidset@nostar} 
\newcommand{\tikzcuboidset@nostar}[1]{\tikzcuboidreset\tikzset{cuboid,#1}}
\newcommand{\tikzcuboidset@star}[1]{\tikzset{cuboid,#1}}
\makeatother

\usetikzlibrary{angles, quotes}

\makeatletter
\newif\ifnobrackets
\renewcommand\@cite[2]{\ifnobrackets\else[\fi{#1\if@tempswa , #2\fi}\ifnobrackets\else]\fi\nobracketsfalse}

\makeatother

\usepackage{lmodern}
\usetikzlibrary{decorations.pathmorphing}
\tikzset{snake it/.style={decorate, decoration=snake}}

\begin{document}

\title{$\ell_1$-norm rank-one symmetric matrix factorization has no spurious second-order stationary points}

\author{
Jiewen GUAN
\thanks{Department of Systems Engineering and Engineering Management, The Chinese University of Hong Kong, Shatin, New Territories, Hong Kong. Email: seemjwguan@gmail.com}
    \and
Anthony Man-Cho SO
\thanks{Department of Systems Engineering and Engineering Management, The Chinese University of Hong Kong, Shatin, New Territories, Hong Kong. Email: manchoso@se.cuhk.edu.hk}
}

\date{\today}

\maketitle

\begin{abstract}
This paper studies the nonsmooth optimization landscape of the $\ell_1$-norm rank-one symmetric matrix factorization problem using tools from second-order variational analysis. Specifically, as the main finding of this paper, we show that any second-order stationary point (and thus local minimizer) of the problem is actually globally optimal. Besides, some other results concerning the landscape of the problem, such as a complete characterization of the set of stationary points, are also developed, which should be interesting in their own rights. Furthermore, with the above theories, we revisit existing results on the generic minimizing behavior of simple algorithms for nonsmooth optimization and showcase the potential risk of their applications to our problem through several examples. 
Our techniques can potentially be applied to analyze the optimization landscapes of
a variety of other more sophisticated nonsmooth learning problems, such as robust low-rank matrix recovery. 

\vspace{0.25cm}


\vspace{0.25cm}


\end{abstract}

\section{Introduction}
Matrix factorization is a fundamental technique for a variety of modern data analytics tasks, such as recommender systems~\cite{koren2009matrix}, network analysis~\cite{luo2021symmetric}, dimensionality reduction~\cite{qian2013robust}, and signal processing~\cite{wu2020hybrid}. Despite their nonconvexity, it has long been understood that the $\ell_2$-norm (a.k.a.\ Frobenius norm) matrix factorization problem as well as its various generalizations (e.g., matrix recovery and completion) enjoy benign landscapes in the following sense: Under certain conditions, their objective functions satisfy the so-called strict saddle property (i.e., each stationary point either corresponds to a local minimizer or the Hessian matrix evaluated there admits a negative eigenvalue; cf., e.g.,~\cite[Definitions~2-3]{zhu2021global}), and they do not have any spurious local minimum
(i.e., all of their local minima are globally optimal); see, e.g.,~\cite{bhojanapalli2016global,ge2016matrix,ge2017no,li2019symmetry,li2019non,zhu2021global}. As an important consequence of these benign properties, many simple optimization algorithms (such as the gradient descent method) are guaranteed to converge to a global optimal solution of the problems almost surely with random initializations; see, e.g.,~\cite{lee2016gradient,jin2017escape}.

Besides the vanilla $\ell_2$-norm matrix factorization discussed above, in the literature, there are also many other robust formulations of matrix factorization for enhancing its noise/outlier tolerance of corrupted data. 
Originated from the well-known fact that the $\ell_1$-norm is more robust than $\ell_2$-norm~\cite{rice1964norms,tibshirani1996regression}, 
the $\ell_1$-norm matrix factorization adopts the sum of elementwise absolute deviations to measure the discrepancy between the factorization and the target matrix. Such an approach has achieved remarkable practical performance in many real-world applications; see, e.g.,~\cite{ke2005robust,eriksson2010efficient,zheng2012practical}. However, despite its tremendous success from a practical point of view, to date, we are not aware of any theoretical results on the optimization landscape of the $\ell_1$-norm formulation that parallels those for its $\ell_2$-norm counterpart mentioned earlier. 
This can be attributed in part 
to the nonsmoothness of the formulation introduced by the $\ell_1$-norm,  
preventing 
the applications of existing analysis techniques used in the smooth case, and raises the very natural question of whether it is still possible to analyze the landscapes of such problems.

In this paper, as a first step towards a thorough understanding of the optimization landscape of $\ell_1$-norm matrix factorization and beyond, we answer, with the help of second-order variational analysis, the above question in the affirmative for the special case of $\ell_1$-norm rank-one symmetric matrix factorization.
Specifically, as the main result of this paper, we show that all second-order stationary points of the problem are actually global minimizers by exhibiting a negative second subderivative at every spurious stationary point (i.e., a stationary point that fails to be globally optimal) along some direction; see Section~\ref{sec:diff-theory} for the formal definitions of these concepts.
This provides a nonsmooth analog of
the results established for the $\ell_2$-norm counterpart and confirms a benign landscape of the problem. Underpinning the main finding is a complete characterization of the set of stationary points of the problem, which is new to the best of our knowledge and also delivers interesting insights. Besides, some other properties regarding the optimization landscape of the problem are also examined, which should be interesting in their own rights. With the landscape of the problem well understood, we then turn to revisit existing results on the generic minimizing behavior of simple algorithms for nonsmooth optimization and showcase the potential risk of their applications to our problem. 

The rest of the paper is organized as follows. We first review and discuss two works that are closely related to ours in Section~\ref{sec:related-work}. Then, we introduce the notation and preliminaries in Section~\ref{sec:notation-prelim}. With the above preparations, we move on to study the optimization landscape of the $\ell_1$-norm rank-one symmetric matrix factorization problem in Section~\ref{sec:main-result}. Afterwards, we revisit existing results on nonsmooth global optimization in Section~\ref{sec:revisit} with the developments in Section~\ref{sec:main-result}. Finally, we conclude this paper in Section~\ref{sec:conclusion} with several future research directions.

\subsection{Related works}\label{sec:related-work}
We are aware of two previous works that have a close connection to ours. The first one is due to Davis et al.~\cite{davis2020nonsmooth}, 
which provides
a complete characterization of the set of stationary points of the population loss of the robust phase retrieval problem~\cite{duchi2019solving}; cf.~\cite[Theorem~5.2]{davis2020nonsmooth}. However, beyond the characterization, no further classifications on the stationary points are given. By contrast, in this work we not only characterize the set of stationary points of the $\ell_1$-norm rank-one symmetric matrix factorization problem but also rule out the existence of spurious second-order stationary points (and thus local minimizers) via the tools from second-order variational analysis. Besides, although  in~\cite{davis2020nonsmooth} an even stronger result characterizing the set of stationary points of an arbitrary convex spectral function of the difference between two rank-one matrices
is developed (cf.~\cite[Corollary~5.12]{davis2020nonsmooth}), we remark that it is not applicable to characterize the set of stationary points of our problem.
This is because even for symmetric matrices, the $\ell_1$-norm fails to be a spectral function, since the spectrum of a symmetric matrix is always invariant under the action of the orthogonal group by conjugation, while the $\ell_1$-norm may not be. (As an aside, we would also like to remark that the $\ell_1$-norm of a matrix (or even a tensor) is actually equal to its nuclear $1$-norm; see~\cite[Proposition~2.6]{chen2020tensor}).
As we shall soon see, the analysis presented here admits a clear distinction from~\cite{davis2020nonsmooth} and is almost from first principles. 

The second work that is closely related to ours is due to Fattahi and Sojoudi~\cite{fattahi2020exact}, 
which studies
the landscape of the nonnegativity-constrained version of our problem. However, 
the analysis therein is totally first-order and the main tool used is the directional derivative. As we shall soon see, for the general version of the problem, there do exist some spurious stationary points with nonnegative directional derivatives along every direction, which suggests that a first-order analysis will be insufficient for obtaining an in-depth understanding of the landscape therearound. Besides, an additional technical assumption on the absence of zero coordinates of the planted ground-truth vector is also made there; see the statement of~\cite[Theorem~8]{fattahi2020exact}. Despite the genericity of such an assumption, it restricts the generality and generalizability of the theory. In this work, we 
overcome the above limitations
and obtain a more general theory through second-order variational analysis.

\section{Notation and preliminaries}\label{sec:notation-prelim}
\subsection{Notation}\label{sec:notations}
Throughout this paper, unless stated otherwise, we adopt lowercase letters (e.g., $x$), boldface lowercase letters (e.g., $\bx=(x_i)$), and boldface capital letters (e.g., $\bX=(x_{i j})$), to denote scalars, vectors, and matrices, respectively. 
Denote $\bbZ$ and $\R$ to be the set of integers and real numbers, respectively, and we fix some $m,n\in\bbZ\cap[1,\infty)$ in subsequent introductions.
We define $[n]:=[1,n]\cap\bbZ$ and $\R_+:=\{x\in\R:x\ge 0\}$, and similar applies to $\R_-$.
For any vector $\bx\in\R^n$, we use 
$$
\|\bx\|:=\sqrt{\sum_{i=1}^n x_i^2}\quad\text{and}\quad\|\bx\|_1:=\sum_{i=1}^n |x_i|
$$
to denote its $\ell_2$-norm and $\ell_1$-norm, respectively. The notation for $\ell_1$-norm also applies for matrices, i.e., 
$$
\|\bX\|_1=\sum_{i=1}^m \sum_{j=1}^n |x_{i j}|,\quad\text{for any matrix $\bX\in\R^{m\times n}$}.
$$
For any matrix $\bX\in\R^{m\times n}$ and index sets $\bbI\subseteq[m]$ and $\bbJ\subseteq[n]$, we use $(x_{i j}:i\in\bbI,\,j\in\bbJ)$ to denote the $|\bbI|\times |\bbJ|$ submatrix of $\bX$ with rows and columns indexed by $\bbI$ and $\bbJ$ respectively, and we also use $\operatorname{vec}(\bX)$ to denote the canonical vectorization of $\bX$, i.e., the concatenation of its columns in order.
Besides, we also use $\R^{n\times n}_{\operatorname{sym}}$ to denote the set of $n\times n$ symmetric matrices.
The $n$-dimensional all-zero vector is denoted by $\bd{0}_n$.
Besides, for $i\in[n]$, we use $\be^n_i$ to denote the $i$-th standard basis vector in $\R^n$.
Moreover, the $n\times n$ identity matrix is denoted by $\bI_n$.
The subscripts and superscripts of these special vectors and matrices indicating the ambient dimensions
are often omitted as long as there is no ambiguity. 
We use $\boxtimes$ to denote the Kronecker product, e.g., 
$$
\bx\boxtimes\by=(x_1\by^{\T},x_2\by^{\T},\dots,x_{m}\by^{\T})^{\T}\in\R^{m n},\quad\text{for any $\bx\in\R^{m}$ and $\by\in\R^{n}$}.
$$
We use $\sum_{j>i}$ and $\sum_{i=1}^{n-1}\sum_{j=i+1}^{n}$ interchangeably whenever $n$ is clear from the contexts, and the same convention applies to other operations besides summation.
For any $\bx,\by\in\R^n$, we use $\langle\bx,\by\rangle:=\sum_{i=1}^n x_i y_i$ to denote the (Frobenius) inner product of $\bx$ and $\by$. The same notation applies to matrices as well, i.e., 
$$
\langle\bX,\bY\rangle=\sum_{i=1}^m \sum_{j=1}^n x_{i j}y_{i j},\quad\text{for any $\bX,\bY\in\R^{m\times n}$}.
$$
Besides, we define for any $\bx\in\R^n$ its support $\operatorname{supp}(\bx):=\{i\in[n]:x_i\neq 0\}$, and for any $\bbJ\subseteq[n]$, we use $\bx_{\bbJ}\in\R^{|\bbJ|}$ to denote the subvector of $\bx$ obtained by throwing all its coordinates outside $\bbJ$ while keeping the orders of the remaining coordinates intact. 
We define the set-valued mapping $\operatorname{Sign}:\R\rightrightarrows\R$ as
$$
    \operatorname{Sign}(x):=
    \begin{dcases}
        \{x/|x|\}, & x\neq 0, \\
        [-1,1], & x=0.
    \end{dcases}
$$
For convenience, we also generalize the above notation for vectors and matrices with an elementwise application; e.g., $\operatorname{Sign}(\bX)=(\operatorname{Sign}(x_{i j}))$ for any $\bX\in\R^{m\times n}$. 
For any set $\mathbb{X}\subseteq\R^n$, we use $\operatorname{int}(\mathbb{X})$ 
to denote its interior,
$\operatorname{dim}(\mathbb{X})$ its dimension, $\operatorname{conv}(\bbX)$ its convex hull, $\operatorname{ext}(\bbX)$ its extreme points, $\operatorname{dist}(\by,\bbX):=\inf\{\|\by-\bx\|:\bx\in\bbX\}$ the distance from $\by\in\R^n$ to $\bbX$, and $|\bbX|$ its cardinality if it is finite. When $\bbX$ is a singleton set, we also identify it with its only element. Suppose that $\bbX$ is in addition convex, then we use 
$$
\N_{\bbX}(\bx):=\{\bv:\langle\bv,\by-\bx\rangle\le 0,\,\forall\,\by\in\bbX\}
$$
to represent its normal cone at $\bx\in\bbX$; cf.~\cite[Theorem~6.9]{rockafellar2009variational}. We also use $\varnothing$ to represent the empty set.
For a random variable $X$, we use $\mathbf{E}(X)$ to denote its expectation. For any function $f:\R^n\rightarrow\R$ that is directionally differentiable, we use 
$$
d{f(\bx)}(\bw):=\lim_{t\searrow 0}\frac{f(\bx+t\bw)-f(\bx)}{t}
$$
to denote its directional derivative at $\bx\in\R^n$ along $\bw\in\R^n$. Besides, for any function $f:\R^n\rightarrow\R$, we say it is $C^p$-smooth for some $p\in\left(\bbZ\cap[1,\infty)\right)\cup\{\infty\}$ if it is $p$-times differentiable with all partial derivatives till $p$-th order continuous, and we say it is $\rho$-weakly-convex for some $\rho\ge 0$ if the function $f+\frac{\rho}{2}\|\cdot\|^2$ is convex. In this paper, some concepts on smooth manifolds will also be used, but we prefer not to introduce them in detail to avoid introducing new notation and terminologies. Instead, we refer the readers to the recent book~\cite{boumal2023introduction} for a comprehensive introduction.

\subsection{Generalized differentiation theory}\label{sec:diff-theory}
In this part, we first introduce three different notions of subdifferentials for locally Lipschitz functions
that will be used throughout the paper. We remark that the first two definitions below are still valid even in absence of the local Lipschitz continuity.

\begin{definition}\label{def:subdiffs}
Given a locally Lipschitz function $f:\R^n\rightarrow\R$ and a point $\bx\in\R^n$:
\begin{itemize}[leftmargin=*]
    \item The Fr\'echet subdifferential~\cite[Definition~8.3(a)]{rockafellar2009variational} of $f$ at $\bx$ is defined as
    $$
        \widehat{\partial} f(\bx):=\left\{\bs \in \R^n: \liminf _{\substack{\by \rightarrow \bx,\, \by\neq\bx}} \frac{f(\by)-f(\bx)-\langle\bs, \by-\bx\rangle}{\|\by-\bx\|} \geq 0\right\}.
    $$
    \item The limiting subdifferential~\cite[Definition~8.3(b)]{rockafellar2009variational} of $f$ at $\bx$ is defined as
    $$
    \partial f(\bx):=\limsup_{\substack{\by\rightarrow\bx ,\, f(\by)\rightarrow f(\bx)}}\widehat{\partial} f(\by),
    $$
    where the $\limsup$ is taken in the sense of 
    Painlev{\'e}-Kuratowski~\cite[Section~5.B]{rockafellar2009variational}.
    \item The Clarke subdifferential~\cite[Theorem~2.5.1]{clarke1990optimization},~\cite[Theorem~9.61]{rockafellar2009variational} of $f$ at $\bx$ is defined as
    $$
    \partial_C f(\bx):=\operatorname{conv}\bigg(\limsup_{\by\rightarrow\bx,\,\by\in\mathbb{D}}\{\nabla f(\by)\}\bigg),
    $$
    where $\mathbb{D}\subseteq\R^n$ is the set on which $f$ is differentiable.
\end{itemize}
\end{definition}

It is well-known that
$$
\widehat{\partial}f(\bx)\subseteq\partial f(\bx)\subseteq\partial_C f(\bx),\quad\text{for any locally Lipschitz function $f:\R^n\rightarrow\R$ and $\bx\in\R^n$};
$$
see, e.g.,~\cite[Proposition~4.3.2(a)]{cui2021modern}. By~\cite[Corollary~8.11]{rockafellar2009variational} as well as
the discussions right after~\cite[Theorem~8.49]{rockafellar2009variational},
we know that for any subdifferentially regular locally Lipschitz function, the above hierarchy 
actually collapses to the lowest level. Since in this paper we will only be dealing with subdifferentially regular locally Lipschitz functions, in what follows, without being explicitly mentioned, we will simply use the terminology ``subdifferential'' with notation ``$\partial$'' to indicate either of the above three constructions. For more properties and relationships of the above three subdifferentials (and beyond), we refer interested readers to~\cite[Section~8]{rockafellar2009variational},~\cite[Section~2]{clarke1990optimization},~\cite[Section~4.3]{cui2021modern}, and~\cite{li2020understanding}. In addition, we would like to recommend the following references~\cite{tian2021hardness,tian2022computing,tian2022finite,tian2023testing,tian2024no} that are helpful for strengthening the understanding of these concepts.

We next introduce the main second-order tool we adopted for nonsmooth landscape analysis.

\begin{definition}\label{def:second-subderivative}
    Given a function $f:\R^n\rightarrow\R$, a point $\bx\in\R^n$, and two arbitrary vectors $\bv,\bw\in\R^n$, the second subderivative~\cite[Definition~13.3]{rockafellar2009variational} of $f$ at $\bx$ for $\bv$ and $\bw$ is defined as
    $$
        d^2{f(\bx;\bv)}(\bw):=\liminf_{\substack{t\searrow 0,\, \bw^{\prime}\rightarrow\bw}}\frac{f(\bx+t\bw^{\prime})-f(\bx)-t\cdot\bv^{\T}\bw^{\prime}}{\frac{1}{2}t^2}.
    $$
\end{definition}

Although Definition~\ref{def:second-subderivative} may look quite strange and unintuitive at a first glance, we remark that it is actually a very natural generalization of the Hessian matrix of a smooth function if it is viewed as a linear operator. Indeed, suppose that $f$ is in addition $C^2$-smooth (or even something weaker as defined in~\cite[Definition~13.1]{rockafellar2009variational}). Then, we have from~\cite[Example~13.8]{rockafellar2009variational} that
$$
d^2{f(\bx;\nabla f(\bx))}(\bw)=\bw^{\T}\nabla^2 f(\bx)\bw,\quad\text{for every $\bw\in\R^n$}.
$$
Besides, it is also known that the second subderivative is positively homogeneous of degree two~\cite[Proposition~13.5]{rockafellar2009variational}, which is the same as a quadratic form. As an aside, we remark that although in Definition~\ref{def:second-subderivative} the designation of $\bv$ can be quite versatile, in this paper, we will only focus on the quantity $d^2{f(\bx;\bd{0})}(\bw)$ at those $\bx$ with $\bd{0}\in\partial f(\bx)$, as this is the only quantity that has something to do with the second-order optimality condition~\cite[Theorem~13.24(a)]{rockafellar2009variational}.

It is worth noting that in this paper some other constructions of second subderivatives will also be (briefly) invoked somewhere, namely: 
\begin{itemize}
    \item $d^2{f(\bx)}(\bw)$, the second subderivative of $f$ at $\bx$ for $\bw$ (no mention of $\bv$)~\cite[Definition~13.3]{rockafellar2009variational}.
    
    \item $d^2{f}(\bx)(\bw;\bz)$, the parabolic subderivative of $f$ at $\bx$ for $\bw$ w.r.t.\ $\bz\in\R^n$~\cite[Definition~13.59]{rockafellar2009variational}.
    
    \item $f^{\prime\prime}(\bx;\bw)$, the (one-sided) second directional derivative of $f$ at $\bx$ along $\bw$~\cite[Equation~13.3]{rockafellar2009variational}.
\end{itemize}
However, as their roles in this paper are not as important as that of the one defined in Definition~\ref{def:second-subderivative}, and for keeping the material succinct and clean and avoiding an overwhelm of consecutive technical definitions, we prefer not to introduce them in further detail. Interested readers may refer to the above pointers for their formal definitions.

Finally, following~\cite[Theorem~13.24(a)]{rockafellar2009variational}, we make precise what do we mean by stationarity.

\begin{definition}
    Given a function $f:\R^n\rightarrow\R$ and a point $\bx\in\R^n$, we say:
    \begin{itemize}
        \item $\bx$ is a stationary point of $f$ if $\bd{0}\in\partial f(\bx)$.
        \item $\bx$ is a second-order stationary point if it is stationary and $d^2{f(\bx;\bd{0})}(\bw)\ge 0$ for all $\bw\in\R^n$.
    \end{itemize}
\end{definition}

\section{Landscape of $\ell_1$-norm rank-one symmetric matrix factorization}\label{sec:main-result}
In this section, we analyze the landscape of the following optimization problem
\begin{equation}\label{eq:subdiff}
    \min\left\{f(\bu):=\frac{1}{2}\|\bu\bu^{\T}-\bu^{\star}{\bu^{\star}}^{\T}\|_1:\bu\in\R^n\right\},
\end{equation}
where $\bu^{\star}\in\R^n$ is some planted ground-truth vector to be recovered by solving the above nonconvex nonsmooth program. We remark that as $f$ is fully amenable (cf.~\cite[Definition~10.23]{rockafellar2009variational}), we know from~\cite[Theorem~10.25(a)]{rockafellar2009variational} that $f$ is subdifferentially regular everywhere. Since we are interested in the stationary points of $f$, we first derive an explicit expression for its subdifferential. Indeed, it is easy to compute that
$$
    \partial f(\bu)=\left\{\left(\frac{\bS+\bS^{\T}}{2}\right)\cdot\bu:\bS\in\operatorname{Sign}(\bu\bu^{\T}-\bu^{\star}{\bu^{\star}}^{\T})\right\}=\left(\operatorname{Sign}(\bu\bu^{\T}-\bu^{\star}{\bu^{\star}}^{\T})\cap\R^{n\times n}_{\operatorname{sym}}\right)\cdot\bu,
$$
where we have used~\cite[Theorem~10.6]{rockafellar2009variational} and the subdifferential regularity of the $\ell_1$-norm implied by its convexity~\cite[Example~7.27]{rockafellar2009variational}. 

\subsection{A complete characterization of the set of stationary points}
As a first but critically important step towards the ultimate goal in this section, we completely characterize the set of stationary points of $f$.

\begin{theorem}\label{thm:char-critical-r1}
    Suppose that $\bu^\star\in\R^n\setminus\{\bd{0}\}$. Then, we have 
    $$
        \{\bu:\bd{0}\in\partial f(\bu)\}=\{\bu:|u_i|\le|u^{\star}_i|,\,i\in[n],\,(\operatorname{Sign}(\bu^{\star}))^{\T}\bu=0\}\cup\{\pm\bu^{\star}\}.
    $$
\end{theorem}

We remark that although the equation $(\operatorname{Sign}(\bu^{\star}))^{\T}\bu=0$ may seem ambiguous at a first glance since $\operatorname{Sign}(u^{\star}_i)=[-1,1]$ if $u^{\star}_i=0$, we know from the constraint $|u_i|\le|u^{\star}_i|$ that such $u_i$ must be zero as well. As a result, $(\operatorname{Sign}(\bu^{\star}))^{\T}\bu$ is actually a singleton set, which by our notation (cf.~Section~\ref{sec:notations}) is identified with its only element, showing the equation is still well-defined.
Before presenting the proof, we mention a useful property 
\begin{equation}\label{eq:property-sign}
    \operatorname{Sign}(x\cdot y)=\operatorname{Sign}(x)\cdot\operatorname{Sign}(y),\quad\text{for every $x,y\in\R$},
\end{equation}
whose verification is routine and thus omitted. It is worth noting that even if some of $x$ and $y$ is zero, such a property is still valid, which is quite interesting. The property (\ref{eq:property-sign}) will be repeatedly used throughout the proof.

\begin{proof}
    Without loss of generality, we may assume that $u^{\star}_i\neq 0$ for every $i\in[n]$. This is because suppose that there is some $i\in[n]$ for which $u^{\star}_i=0$. Then, whenever $u_i\neq 0$, for any $\bg\in\partial f(\bu)$, we know
    $$
        g_i=\left(\operatorname{Sign}(u_i\cdot\bu)\right)^{\T}\bu=\operatorname{Sign}(u_i)\cdot\|\bu\|_1\neq 0,
    $$
    and thus such $\bu$ can never be stationary. However, whenever $u_i=0$, we can then equivalently study the dimension-reduced version of the problem where the $i$-th coordinates of $\bu$ and $\bu^{\star}$ are both removed. This is feasible because in this case the $i$-th coordinate of $\partial f(\bu)$ must contain zero, and the other coordinates totally have nothing to do with $u_i$ and $u^{\star}_i$. Besides, we may repeat the above procedure until the problem has been reduced such that all coordinates of $\bu^{\star}$ are nonzero, which is right the setting as assumed.
    
    Before starting the formal proof, we make the following definitions
    $$
        \bbJ_{>}(\bu):=\{i\in[n]:|u_i|>|u_i^{\star}|\},~\bbJ_{=}(\bu):=\{i\in[n]:|u_i|=|u_i^{\star}|\},~\bbJ_{<}(\bu):=\{i\in[n]:|u_i|<|u_i^{\star}|\}.
    $$
    It is clear from the definitions that $[n]=\bbJ_{>}(\bu)\cup\bbJ_{=}(\bu)\cup\bbJ_{<}(\bu)$. Besides, the above indices can be further partitioned into finer granularity as
    $$
        \bbJ_{>}(\bu)=\bbJ_{>,\sim}(\bu)\cup\bbJ_{>,\centernot{\sim}}(\bu),~\bbJ_{=}(\bu)=\bbJ_{=,\sim}(\bu)\cup\bbJ_{=,\centernot{\sim}}(\bu),~\bbJ_{<}(\bu)=\bbJ_{<,\sim}(\bu)\cup\bbJ_{<,0}(\bu)\cup\bbJ_{<,\centernot{\sim}}(\bu),
    $$
    where
    $$
        \bbJ_{>,\sim}(\bu):=\{i\in\bbJ_{>}:\operatorname{Sign}(u_i)=\operatorname{Sign}(u^{\star}_i)\},~\bbJ_{>,\centernot{\sim}}(\bu):=\{i\in\bbJ_{>}:\operatorname{Sign}(u_i)=-\operatorname{Sign}(u^{\star}_i)\},
    $$
    $$
        \bbJ_{=,\sim}(\bu):=\{i\in\bbJ_{=}:\operatorname{Sign}(u_i)=\operatorname{Sign}(u^{\star}_i)\},~\bbJ_{=,\centernot{\sim}}(\bu):=\{i\in\bbJ_{=}:\operatorname{Sign}(u_i)=-\operatorname{Sign}(u^{\star}_i)\},
    $$
    $$
        \bbJ_{<,\sim}(\bu):=\{i\in\bbJ_{<}:\operatorname{Sign}(u_i)=\operatorname{Sign}(u^{\star}_i)\},~\bbJ_{<,\centernot{\sim}}(\bu):=\{i\in\bbJ_{<}:\operatorname{Sign}(u_i)=-\operatorname{Sign}(u^{\star}_i)\},
    $$
    and 
    $$
        \bbJ_{<,0}(\bu):=\{i\in\bbJ_{<}:u_i=0\}.
    $$
    With the above indices defined, we next make another assumption that for any given $\bu$, its coordinates together with the ones of $\bu^{\star}$ have already been simultaneously permuted\footnote{There could be multiple possible permutations for our purposes, but anyone is fine.} such that $\bbJ_{>,\sim}(\bu)$, $\bbJ_{>,\centernot{\sim}}(\bu)$, $\bbJ_{=,\sim}(\bu)$, $\bbJ_{=,\centernot{\sim}}(\bu)$, $\bbJ_{<,\sim}(\bu)$, $\bbJ_{<,0}(\bu)$, and $\bbJ_{<,\centernot{\sim}}(\bu)$ form an \textit{ordered} partition of $[n]$. This assumption is also without loss of generality (and is only to simplify notation and ease understanding and presentation), because as we shall soon see, all of the subsequent characterizations and their deductions are invariant under simultaneous permutations of $\bu$ and $\bu^{\star}$. With the above preparations, we next embark on the main proof, which is divided into the following seven cases. In each case, we will either refute the possibility or find an equivalent condition for $\bu$ to be stationary.
    \begin{itemize}[leftmargin=*]
        \item $[n]=\bbJ_{>}(\bu)$: In this case, we know $\Signuuuu=\operatorname{Sign}(\bu\bu^{\T})=\operatorname{Sign}(\bu)\cdot(\operatorname{Sign}(\bu))^{\T}$, where the second equality directly follows from (\ref{eq:property-sign}). This, together with (\ref{eq:subdiff}), further implies that for any (and actually the only) $\bZ\in\partial f(\bu)=\operatorname{Sign}(\bu\bu^{\T}-\bu^{\star}{\bu^{\star}}^{\T})\cap\R^{n\times n}_{\operatorname{sym}}$, we have
        $$
            \bZ\bu=\operatorname{Sign}(\bu)\cdot(\operatorname{Sign}(\bu))^{\T}\bu=\|\bu\|_1\cdot\operatorname{Sign}(\bu)\neq 0,
        $$
        which implies $\bu$ can never be stationary.
        
        \item $[n]=\bbJ_{=}(\bu)$: In this case, suppose that $\bbJ_{=}(\bu)$ equals either $\bbJ_{=,\sim}(\bu)$ or $\bbJ_{=,\centernot{\sim}}(\bu)$. Then, we know $\bu$ must be either of the ground-truths $\pm\bu^{\star}$.
        Conversely, since $\pm\bu^{\star}$ are clearly global minimizers of $f$, by the generalized Fermat's rule~\cite[Theorem~10.1]{rockafellar2009variational} we know they are stationary points of $f$ as well. Otherwise, since  
        $$
            \Signuiujuiuj=-\operatorname{Sign}(u_i^{\star} u_j^{\star}),\quad\text{for any $i\in\bbJ_{=,\sim}(\bu)$ and $j\in\bbJ_{=,\centernot{\sim}}(\bu)$},
        $$
        we know the following partition holds for $\Signuuuu$
        $$
            \begin{pmatrix}
                [-1,1]^{|\bbJ_{=,\sim}(\bu)|\times |\bbJ_{=,\sim}(\bu)|} & \left(-\operatorname{Sign}(u_i^{\star} u_j^{\star}):i\in\bbJ_{=,\sim}(\bu),\,j\in\bbJ_{=,\centernot{\sim}}(\bu)\right) \\ 
                \left(-\operatorname{Sign}(u_i^{\star} u_j^{\star}):i\in\bbJ_{=,\centernot{\sim}}(\bu),\,j\in\bbJ_{=,\sim}(\bu)\right) & [-1,1]^{|\bbJ_{=,\centernot{\sim}}(\bu)|\times |\bbJ_{=,\centernot{\sim}}(\bu)|}
            \end{pmatrix}.
        $$
        As a result, suppose that $\bu$ is stationary, then we must have the following system
        $$
            \begin{dcases}
                \sum_{j\in\bbJ_{=,\sim}(\bu)}|u_j|\ge\left|\sum_{j\in\bbJ_{=,\centernot{\sim}}(\bu)}\left(-\operatorname{Sign}(u_i^{\star} u_j^{\star})\right)\cdot u_j\right|=\sum_{j\in\bbJ_{=,\centernot{\sim}}(\bu)}|u_j|, & i\in\bbJ_{=,\sim}(\bu), \\
                \sum_{j\in\bbJ_{=,\centernot{\sim}}(\bu)}|u_j|\ge\left|\sum_{j\in\bbJ_{=,\sim}(\bu)}\left(-\operatorname{Sign}(u_i^{\star} u_j^{\star})\right)\cdot u_j\right|=\sum_{j\in\bbJ_{=,\sim}(\bu)}|u_j|, & i\in\bbJ_{=,\centernot{\sim}}(\bu),
            \end{dcases}
        $$
        which together imply
        $$
            \sum_{j\in\bbJ_{=,\sim}(\bu)}|u_j|=\sum_{j\in\bbJ_{=,\centernot{\sim}}(\bu)}|u_j|\iff(\operatorname{Sign}(\bu^{\star}))^{\T}\bu=0.
        $$
        Conversely, suppose that we have $(\operatorname{Sign}(\bu^{\star}))^{\T}\bu=0$. Then, by letting $\bZ=-\operatorname{Sign}(\bu^{\star}{\bu^{\star}}^{\T})\in\R^{n\times n}_{\operatorname{sym}}$, we know $\bZ\in\Signuuuu\cap\R^{n\times n}_{\operatorname{sym}}$ with $\bZ\bu=0$, and $\bu$ is thus stationary.
        
        \item $[n]=\bbJ_{<}(\bu)$: In this case, we know $\Signuuuu=-\operatorname{Sign}(\bu^{\star}{\bu^{\star}}^{\T})\in\R^{n\times n}_{\operatorname{sym}}$. As a result, suppose that $\bu$ is stationary, then we must have
        $$
            \operatorname{Sign}(\bu^{\star}{\bu^{\star}}^{\T})\cdot\bu=\bd{0}\iff(\operatorname{Sign}(\bu^{\star}))^{\T}\bu=0.
        $$
        The converse is also true, as can be easily seen.
        
    \item $[n]=\bbJ_{>}(\bu)\cup\bbJ_{=}(\bu)$ with $|\bbJ_{>}(\bu)|,|\bbJ_{=}(\bu)|\neq 0$: In this case, we observe $\Signuuuu$ has the following structure
    $$
        \begin{pmatrix}
            \Big(\operatorname{Sign}(u_i u_j):i\in\bbJ_{>}(\bu),\,j\in\bbJ_{>}(\bu)\Big) & \Big(\operatorname{Sign}(u_i u_j):i\in\bbJ_{>}(\bu),\,j\in\bbJ_{=}(\bu)\Big) \\
            \Big(\operatorname{Sign}(u_i u_j):i\in\bbJ_{=}(\bu),\,j\in\bbJ_{>}(\bu)\Big) & \bbT\subseteq [-1,1]^{|\bbJ_{=}(\bu)|\times|\bbJ_{=}(\bu)|}
        \end{pmatrix},
    $$
    where $\bbT$ is something that is not of interest. As a result, we have for any $\bg\in\partial f(\bu)$ that
    $$
        g_i=\operatorname{Sign}(u_i)\cdot\sum_{j=1}^n \operatorname{Sign}(u_j)\cdot u_j=\operatorname{Sign}(u_i)\cdot\|\bu\|_1\neq 0,\quad\text{for any $i\in\bbJ_{>}(\bu)$},
    $$
    and $\bu$ can therefore never be stationary.

    \item $[n]=\bbJ_{=}(\bu)\cup\bbJ_{<}(\bu)$ with $|\bbJ_{=}(\bu)|,|\bbJ_{<}(\bu)|\neq 0$: In this case, we observe $\Signuuuu$ has the following structure
    $$
        \begin{pmatrix}
            \left(\operatorname{Sign}(u_i u_j -u_i^{\star} u_j^{\star}):i\in\bbJ_{=}(\bu),\,j\in\bbJ_{=}(\bu)\right) & \left(-\operatorname{Sign}(u_i^{\star} u_j^{\star}):i\in\bbJ_{=}(\bu),\,j\in\bbJ_{<}(\bu)\right) \\
            \left(-\operatorname{Sign}(u_i^{\star} u_j^{\star}):i\in\bbJ_{<}(\bu),\,j\in\bbJ_{=}(\bu)\right) & \left(-\operatorname{Sign}(u_i^{\star} u_j^{\star}):i\in\bbJ_{<}(\bu),\,j\in\bbJ_{<}(\bu)\right)
        \end{pmatrix}.
    $$
    Suppose that $\bu$ is stationary, then we must have
    $$
        -\operatorname{Sign}(u_i^{\star})\cdot(\operatorname{Sign}(\bu^{\star}))^{\T}\bu=0\iff (\operatorname{Sign}(\bu^{\star}))^{\T}\bu=0,\quad\text{for all $i\in\bbJ_{<}(\bu)$}.
    $$
    We next claim the following set inclusion
    $$
        \left(-\operatorname{Sign}(u_i^{\star} u_j^{\star}):i\in\bbJ_{=}(\bu),\,j\in\bbJ_{=}(\bu)\right)\in\left(\operatorname{Sign}(u_i u_j -u_i^{\star} u_j^{\star}):i\in\bbJ_{=}(\bu),\,j\in\bbJ_{=}(\bu)\right),
    $$
    which will validate the converse direction, as desired. This is done in the following discussions:
    \begin{itemize}[leftmargin=*]
        \item Suppose that $\bbJ_{=,\sim}(\bu)=\bbJ_{=}(\bu)$. Then
        $$
            \left(\operatorname{Sign}(u_i u_j -u_i^{\star} u_j^{\star}):i\in\bbJ_{=}(\bu),\,j\in\bbJ_{=}(\bu)\right)=[-1,1]^{|\bbJ_{=}(\bu)|\times |\bbJ_{=}(\bu)|},
        $$
        which covers the desired element vacuously.
        \item Suppose that $\bbJ_{=,\centernot{\sim}}(\bu)=\bbJ_{=}(\bu)$. Then
        $$
            \left(\operatorname{Sign}(u_i u_j -u_i^{\star} u_j^{\star}):i\in\bbJ_{=}(\bu),\,j\in\bbJ_{=}(\bu)\right)=\left(-\operatorname{Sign}(u_i^{\star} u_j^{\star}):i\in\bbJ_{=}(\bu),\,j\in\bbJ_{=}(\bu)\right),
        $$
        which is exactly the interested element, and the desired claim thus holds true.
        \item Suppose that neither of them happens. Then, the interested set has the following structure
        $$
            \begin{pmatrix}
                [-1,1]^{|\bbJ_{=,\sim}(\bu)|\times|\bbJ_{=,\sim}(\bu)|} & \left(-\operatorname{Sign}(u_i^{\star} u_j^{\star}):i\in\bbJ_{=,\sim}(\bu),\,j\in\bbJ_{=,\centernot{\sim}}(\bu)\right) \\
                \left(-\operatorname{Sign}(u_i^{\star} u_j^{\star}):i\in\bbJ_{=,\centernot{\sim}}(\bu),\,j\in\bbJ_{=,\sim}(\bu)\right)  & [-1,1]^{|\bbJ_{=,\centernot{\sim}}(\bu)|\times|\bbJ_{=,\centernot{\sim}}(\bu)|}
            \end{pmatrix},
        $$
        which, serving as an interpolation between the above two cases, is also desirable.
    \end{itemize}
    \item $[n]=\bbJ_{>}(\bu)\cup\bbJ_{<}(\bu)$ with $|\bbJ_{>}(\bu)|,|\bbJ_{<}(\bu)|\neq 0$: In this case, the situation becomes even more complicated. Therefore, we shall first exempt $\bbJ_{<,0}(\bu)$ from the analysis and assume that $|\bbJ_{<,0}(\bu)|=0$ in what follows; this will be revisited later. Let us divide the analysis into the following points:
    \begin{itemize}[leftmargin=*]
        \item Suppose that $|\bbJ_{>,\sim}(\bu)|,|\bbJ_{>,\centernot{\sim}}(\bu)|,|\bbJ_{<,\sim}(\bu)|,|\bbJ_{<,\centernot{\sim}}(\bu)|\neq 0$. Then, we observe $\Signuuuu$ has the following structure
    $$
        \begin{pmatrix}
            \Big(\operatorname{Sign}(u_i u_j):i\in\bbJ_{>}(\bu),\,j\in\bbJ_{>}(\bu)\Big) & \bbT\subseteq[-1,1]^{|\bbJ_{>}(\bu)|\times|\bbJ_{<}(\bu)|} \\
            \bbH\subseteq[-1,1]^{|\bbJ_{<}(\bu)|\times|\bbJ_{>}(\bu)|} & \left(-\operatorname{Sign}(u_i^{\star} u_j^{\star}):i\in\bbJ_{<}(\bu),\,j\in\bbJ_{<}(\bu)\right)
        \end{pmatrix},
    $$
    where $\bbT$ and $\bbH$ further have the following structures, respectively
    $$
        \begin{pmatrix}
            \bbT_{1 1}\subseteq[-1,1]^{|\bbJ_{>,\sim}(\bu)|\times |\bbJ_{<,\sim}(\bu)|} & \left(-\operatorname{Sign}(u_i^{\star} u_j^{\star}):i\in\bbJ_{>,\sim}(\bu),\,j\in\bbJ_{<,\centernot{\sim}}(\bu)\right)\\
            \left(-\operatorname{Sign}(u_i^{\star} u_j^{\star}):i\in\bbJ_{>,\centernot{\sim}}(\bu),\,j\in\bbJ_{<,\sim}(\bu)\right) & \bbT_{2 2}\subseteq[-1,1]^{|\bbJ_{>,\centernot{\sim}}(\bu)|\times |\bbJ_{<,\centernot{\sim}}(\bu)|}
        \end{pmatrix},
    $$
    and
    $$
        \begin{pmatrix}
            \bbH_{1 1}\subseteq[-1,1]^{|\bbJ_{<,\sim}(\bu)|\times |\bbJ_{>,\sim}(\bu)|} & \left(-\operatorname{Sign}(u_i^{\star} u_j^{\star}):i\in\bbJ_{<,\sim}(\bu),\,j\in\bbJ_{>,\centernot{\sim}}(\bu)\right)\\
            \left(-\operatorname{Sign}(u_i^{\star} u_j^{\star}):i\in\bbJ_{<,\centernot{\sim}}(\bu),\,j\in\bbJ_{>,\sim}(\bu)\right) & \bbH_{2 2}\subseteq[-1,1]^{|\bbJ_{<,\centernot{\sim}}(\bu)|\times |\bbJ_{>,\centernot{\sim}}(\bu)|}
        \end{pmatrix}.
    $$
    Suppose that $\bu$ is stationary. Then, we must have the following system
    $$
        \begin{dcases}
            \sum_{j\in\bbJ_{<,\sim}(\bu)}|u_j|\ge\left|\operatorname{Sign}(u_i^{\star})\cdot\left(\sum_{j\in\bbJ_{>,\sim}(\bu)}|u_j|+\sum_{j\in\bbJ_{>,\centernot{\sim}}(\bu)}|u_j|+\sum_{j\in\bbJ_{<,\centernot{\sim}}(\bu)}|u_j|\right)\right|, & i\in\bbJ_{>,\sim}(\bu), \\
            \sum_{j\in\bbJ_{<,\centernot{\sim}}(\bu)}|u_j|\ge\left|-\operatorname{Sign}(u_i^{\star})\cdot\left(\sum_{j\in\bbJ_{>,\sim}(\bu)}|u_j|+\sum_{j\in\bbJ_{>,\centernot{\sim}}(\bu)}|u_j|+\sum_{j\in\bbJ_{<,\sim}(\bu)}|u_j|\right)\right|, & i\in\bbJ_{>,\centernot{\sim}}(\bu),
        \end{dcases}
    $$
    which together imply
    $$
        \sum_{j\in\bbJ_{>,\sim}(\bu)}|u_j|+\sum_{j\in\bbJ_{>,\centernot{\sim}}(\bu)}|u_j|=0,
    $$
    and thus $u_i=0$ for every $i\in\bbJ_{>}(\bu)$, a contradiction. Therefore, $\bu$ can never be stationary.
    
    \item Suppose that $|\bbJ_{>,\sim}(\bu)|,|\bbJ_{>,\centernot{\sim}}(\bu)|\neq 0$ but $\bbJ_{<}(\bu)=\bbJ_{<,\sim}(\bu)$. Then, we observe $\Signuuuu$ has the following structure
    $$
        \begin{pmatrix}
            \multicolumn{3}{c}{\bbT\subseteq [-1,1]^{|\bbJ_{>,\sim}(\bu)|\times n}} \\
            \multicolumn{2}{c}{\Big(\operatorname{Sign}(u_i u_j):i\in\bbJ_{>,\centernot{\sim}}(\bu),\,j\in\bbJ_{>}(\bu)\Big)} & \left(-\operatorname{Sign}(u_i^{\star} u_j^{\star}):i\in\bbJ_{>,\centernot{\sim}}(\bu),\,j\in\bbJ_{<}(\bu)\right) \\
            \multicolumn{3}{c}{\bbH\subseteq [-1,1]^{|\bbJ_{<}(\bu)|\times n}}
        \end{pmatrix},
    $$
    where $\bbT$ and $\bbH$ are something unimportant. Therefore, for $\bu$ to be stationary, we must require
    $$
        0=\operatorname{Sign}(u_i)\cdot\sum_{j\in\bbJ_{>}(\bu)}|u_j|+\operatorname{Sign}(u_i)\cdot\sum_{j\in\bbJ_{<}(\bu)}|u_j|=\operatorname{Sign}(u_i)\cdot\|\bu\|_1,\quad\text{for any $i\in\bbJ_{>,\centernot{\sim}}(\bu)$},
    $$
    which is impossible. Therefore, $\bu$ can never be stationary. As the analysis for the case where $|\bbJ_{>,\sim}(\bu)|,|\bbJ_{>,\centernot{\sim}}(\bu)|\neq 0$ but $\bbJ_{<}(\bu)=\bbJ_{<,\centernot{\sim}}(\bu)$ is totally symmetric to this one, it is omited in the sequel for succinctness.
    
    \item Suppose that $|\bbJ_{<,\sim}(\bu)|,|\bbJ_{<,\centernot{\sim}}(\bu)|\neq 0$ but $\bbJ_{>}(\bu)=\bbJ_{>,\sim}(\bu)$. Then, we observe $\Signuuuu$ has the following structure
    $$
        \begin{pmatrix}
            \Big(\operatorname{Sign}(u_i u_j):i\in\bbJ_{>}(\bu),\,j\in\bbJ_{>}(\bu)\Big) & \bbT\subseteq[-1,1]^{|\bbJ_{>}(\bu)|\times|\bbJ_{<}(\bu)|} \\
            \bbH\subseteq[-1,1]^{|\bbJ_{<}(\bu)|\times|\bbJ_{>}(\bu)|} & \left(-\operatorname{Sign}(u_i^{\star} u_j^{\star}):i\in\bbJ_{<}(\bu),\,j\in\bbJ_{<}(\bu)\right)
        \end{pmatrix},
    $$
    where $\bbT$ and $\bbH$ further have the following structures, respectively
    $$
        \begin{pmatrix}
            \bbT_{1}\subseteq[-1,1]^{|\bbJ_{>}(\bu)|\times |\bbJ_{<,\sim}(\bu)|} & \left(-\operatorname{Sign}(u_i^{\star} u_j^{\star}):i\in\bbJ_{>}(\bu),\,j\in\bbJ_{<,\centernot{\sim}}(\bu)\right)
        \end{pmatrix},
    $$
    and
    $$
        \begin{pmatrix}
            \bbH_{1}\subseteq[-1,1]^{|\bbJ_{<,\sim}(\bu)|\times |\bbJ_{>}(\bu)|}\\
            \left(-\operatorname{Sign}(u_i^{\star} u_j^{\star}):i\in\bbJ_{<,\centernot{\sim}}(\bu),\,j\in\bbJ_{>}(\bu)\right)
        \end{pmatrix}.
    $$
    Suppose that $\bu$ is stationary. 
    Then, it must hold that
    $$
        -\operatorname{Sign}(u_i^{\star})\cdot(\operatorname{Sign}(\bu^{\star}))^{\T}\bu=0\iff(\operatorname{Sign}(\bu^{\star}))^{\T}\bu=0,\quad\text{for any $i\in\bbJ_{<,\centernot{\sim}}(\bu)$}.
    $$
    Besides, the following system must also hold true
    $$
        \begin{dcases}
            \sum_{j\in\bbJ_{<,\sim}(\bu)}|u_j|\ge\left|\operatorname{Sign}(u_i)\cdot\left(\sum_{j\in\bbJ_{>}(\bu)}|u_j|+\sum_{j\in\bbJ_{<,\centernot{\sim}}(\bu)}|u_j|\right)\right|,& i\in\bbJ_{>}(\bu), \\
            \sum_{j\in\bbJ_{>}(\bu)}|u_j|\ge\left|-\operatorname{Sign}(u_i^{\star})\cdot\left(\sum_{j\in\bbJ_{<,\sim}(\bu)}|u_j|-\sum_{j\in\bbJ_{<,\centernot{\sim}}(\bu)}|u_j|\right)\right|,& i\in\bbJ_{<,\sim}(\bu).
        \end{dcases}
    $$
    It is clear that the second equation holds true vacuously, as we already know
    $$
        0=(\operatorname{Sign}(\bu^{\star}))^{\T}\bu=\sum_{j\in\bbJ_{>}(\bu)}|u_j|+\sum_{j\in\bbJ_{<,\sim}(\bu)}|u_j|-\sum_{j\in\bbJ_{<,\centernot{\sim}}(\bu)}|u_j|.
    $$
    However, the first equation together with the above identity further implies
    $$
        -\sum_{j\in\bbJ_{>}(\bu)}|u_j|=\sum_{j\in\bbJ_{<,\sim}(\bu)}|u_j|-\sum_{j\in\bbJ_{<,\centernot{\sim}}(\bu)}|u_j|\ge\sum_{j\in\bbJ_{>}(\bu)}|u_j|\implies \sum_{j\in\bbJ_{>}(\bu)}|u_j|\le 0,
    $$
    and therefore $u_i=0$ for every $i\in\bbJ_{>}(\bu)$, a contradiction. Thus, $\bu$ can never be stationary.
    Due to the same reason, the symmetric case is also waived from our discussions.
    
    \item Suppose that $\bbJ_{>}(\bu)=\bbJ_{>,\sim}(\bu)$ and $\bbJ_{<}(\bu)=\bbJ_{<,\sim}(\bu)$ simultaneously. Then, it is easy to observe that $\Signuuuu$ has the following structure
    $$
        \begin{pmatrix}
            \Big(\operatorname{Sign}(u_i u_j):i\in\bbJ_{>}(\bu),\,j\in\bbJ_{>}(\bu)\Big) & \bbT\subseteq[-1,1]^{|\bbJ_{>}(\bu)|\times|\bbJ_{<}(\bu)|} \\
            \bbH\subseteq[-1,1]^{|\bbJ_{<}(\bu)|\times|\bbJ_{>}(\bu)|} & \left(-\operatorname{Sign}(u_i^{\star} u_j^{\star}):i\in\bbJ_{<}(\bu),\,j\in\bbJ_{<}(\bu)\right)
        \end{pmatrix},
    $$
    where $\bbT$ and $\bbH$ are something not of interest. Suppose that $\bu$ is stationary. Then, we must have the following system
    $$
        \begin{dcases}
            \sum_{j\in\bbJ_{<}(\bu)}|u_j|\ge\sum_{j\in\bbJ_{>}(\bu)}|u_j|, & i\in\bbJ_{>}(\bu), \\
            \sum_{j\in\bbJ_{>}(\bu)}|u_j|\ge\sum_{j\in\bbJ_{<}(\bu)}|u_j|, & i\in\bbJ_{<}(\bu),
        \end{dcases}
        \implies \sum_{j\in\bbJ_{>}(\bu)}|u_j|=\sum_{j\in\bbJ_{<}(\bu)}|u_j|.
    $$
    Therefore, the only possible element $\bZ\in\Signuuuu\supseteq\operatorname{Sign}(\bu\bu^{\T}-\bu^{\star}{\bu^{\star}}^{\T})\cap\R^{n\times n}_{\operatorname{sym}}$ that satisfies $\bZ\bu=\bd{0}$ has to be the following matrix
    $$
        \begin{pmatrix}
            \Big(\operatorname{Sign}(u_i u_j):i\in\bbJ_{>}(\bu),\,j\in\bbJ_{>}(\bu)\Big) & \Big(-\operatorname{Sign}(u_i u_j):i\in\bbJ_{>}(\bu),\,j\in\bbJ_{<}(\bu)\Big) \\
            \left(\operatorname{Sign}(u_i^{\star} u_j^{\star}):i\in\bbJ_{<}(\bu),\,j\in\bbJ_{>}(\bu)\right) & \left(-\operatorname{Sign}(u_i^{\star} u_j^{\star}):i\in\bbJ_{<}(\bu),\,j\in\bbJ_{<}(\bu)\right)
        \end{pmatrix}\notin\R^{n\times n}_{\operatorname{sym}}.
    $$
    Since $\operatorname{Sign}(\bu\bu^{\T}-\bu^{\star}{\bu^{\star}}^{\T})\cap\R^{n\times n}_{\operatorname{sym}}$ contains only symmetric matrices, this leads to a contradiction. As a result, $\bu$ can never be stationary. We can apply the above analysis symmetrically to study the case where $\bbJ_{>}(\bu)=\bbJ_{>,\centernot{\sim}}(\bu)$ and $\bbJ_{<}(\bu)=\bbJ_{<,\centernot{\sim}}(\bu)$ simultaneously, and it is thus omitted to avoid redundancy.
    \item Suppose that $\bbJ_{>}(\bu)=\bbJ_{>,\sim}(\bu)$ and $\bbJ_{<}(\bu)=\bbJ_{<,\centernot{\sim}}(\bu)$ simultaneously. Then, it can be easily checked that $\Signuuuu$ equals the following matrix
    $$
        \begin{pmatrix}
            \Big(\operatorname{Sign}(u_i u_j):i\in\bbJ_{>}(\bu),\,j\in\bbJ_{>}(\bu)\Big) & \Big(-\operatorname{Sign}(u_i^{\star} u_j^{\star}):i\in\bbJ_{>}(\bu),\,j\in\bbJ_{<}(\bu)\Big) \\
            \left(-\operatorname{Sign}(u_i^{\star} u_j^{\star}):i\in\bbJ_{<}(\bu),\,j\in\bbJ_{>}(\bu)\right) & \left(-\operatorname{Sign}(u_i^{\star} u_j^{\star}):i\in\bbJ_{<}(\bu),\,j\in\bbJ_{<}(\bu)\right)
        \end{pmatrix}.
    $$
    Since we know $-\operatorname{Sign}(u_i^{\star} u_j^{\star})=\operatorname{Sign}(u_i u_j)$ for any $i\in\bbJ_{>}(\bu),\,j\in\bbJ_{<}(\bu)$, it follows that $\bu$ can never be stationary due to the same reason as discussed in the case where $[n]=\bbJ_{>}(\bu)\cup\bbJ_{=}(\bu)$ with $|\bbJ_{>}(\bu)|,|\bbJ_{=}(\bu)|\neq 0$. The symmetric case should be clear as well.
    \end{itemize}
    As the above discussions cover all possible cases where $[n]=\bbJ_{>}(\bu)\cup\bbJ_{<}(\bu)$ with $|\bbJ_{>}(\bu)|,|\bbJ_{<}(\bu)|\allowbreak\neq 0$ and $|\bbJ_{<,0}(\bu)|=0$, all it remains is to discuss how to handle the situation where $|\bbJ_{<,0}(\bu)|>0$. Suppose this happens. Since $u_i=0$ for every $i\in\bbJ_{<,0}(\bu)$, as a necessary condition for $\bu$ to be stationary, we must require that there exists some $\bZ\in\operatorname{Sign}(\bu\bu^{\T}-\bu^{\star}{\bu^{\star}}^{\T})\cap\R^{n\times n}_{\operatorname{sym}}$, which, after removing its rows and columns corresponding to $\bbJ_{<,0}(\bu)$, has the dimension-reduced version of $\bu$ with all zeros removed, namely $\bu_{[n]\setminus\bbJ_{<,0}(\bu)}$, inside its kernel. This has reduced the situation to the one where $|\bbJ_{<,0}(\bu)|=0$, which we have just studied before.
    (It is worth noting that when $\bbJ_{<}(\bu)=\bbJ_{<,0}(\bu)$ the problem will be reduced to the one where $[n]=\bbJ_{>}(\bu)$.) 
    Therefore, we can conclude $\bu$ can never be stationary in this case as well.

    \item $[n]=\bbJ_{>}(\bu)\cup\bbJ_{=}(\bu)\cup\bbJ_{<}(\bu)$ with $|\bbJ_{>}(\bu)|,|\bbJ_{=}(\bu)|,|\bbJ_{<}(\bu)|\neq 0$: 
    Actually, the situation in this case is very similar to that of the previous one where $[n]=\bbJ_{>}(\bu)\cup\bbJ_{<}(\bu)$ with $|\bbJ_{>}(\bu)|,|\bbJ_{<}(\bu)|\neq 0$, but it unfortunately becomes too complicated and tedious to be presentable. Hence, we directly conclude $\bu$ can never be stationary in this case as well and leave the details to interested readers; with the machinery developed so far, this should be as easy as a routine exercise.
    \end{itemize}
    Combining the above discussions reveals that the only possibility for $\bu$ to be stationary is to either have $\bu=\pm\bu^{\star}$ or simultaneously satisfy $|u_i|\le|u^{\star}_i|$ for every $i\in[n]$ and $(\operatorname{Sign}(\bu^{\star}))^{\T}\bu=0$, both of which are invariant under simultaneous permutations of $\bu$ and $\bu^{\star}$. Besides, the corresponding converse implications of the above conditions have also been verified. These, together with the remarks made at the beginning of the proof, complete the whole proof as desired.
\end{proof}

It is worth noting that besides for our purposes, Theorem~\ref{thm:char-critical-r1} can also facilitate the analysis of $f$ from many other perspectives, such as characterizing the Kurdyka-Łojasiewicz exponent at each of its stationary points, which has not been covered by a standard application of existing calculi~\cite{li2018calculus,yu2022kurdyka}. In the sequel, we call $\bu\in\R^n$ a spurious stationary point if it is stationary but different from any of the ground-truths $\pm\bu^{\star}$. We remark that the set of stationary points of $f$ as revealed in Theorem~\ref{thm:char-critical-r1} is clearly much more intricate than its $\ell_2$-norm counterpart, which has only three (and thus finitely many) elements, namely $\pm\bu^{\star}$ and $\bd{0}$ (see, e.g.,~\cite[Theorem~2]{chi2019nonconvex} and~\cite[Theorem~3]{li2019symmetry}). 
Despite this, it is worth noting that
the separation between the spurious stationary points and ground-truths of $f$ can still be rather larger, at least when $\bu^{\star}$ is standard Gaussian.

\begin{remark}
    Suppose that $\bu^{\star}$ is standard Gaussian. Then, standard computation shows
    $$
        \mathbf{E}\left(\operatorname{dist}(\pm\bu^{\star},\{\bu:(\operatorname{Sign}(\bu^{\star}))^{\T}\bu=0\})\right)=\mathbf{E}\left(\frac{(\operatorname{Sign}(\bu^{\star}))^{\T}\bu^{\star}}{\|\operatorname{Sign}(\bu^{\star})\|}\right)=\mathbf{E}\left(\frac{\|\bu^{\star}\|_1}{\sqrt{n}}\right)=\sqrt{\frac{2n}{\pi}}.
    $$
\end{remark}

\begin{wrapfigure}[11]{r}{0.4\linewidth}
\vspace{-1em}
    \centering
    \includegraphics[width=\linewidth]{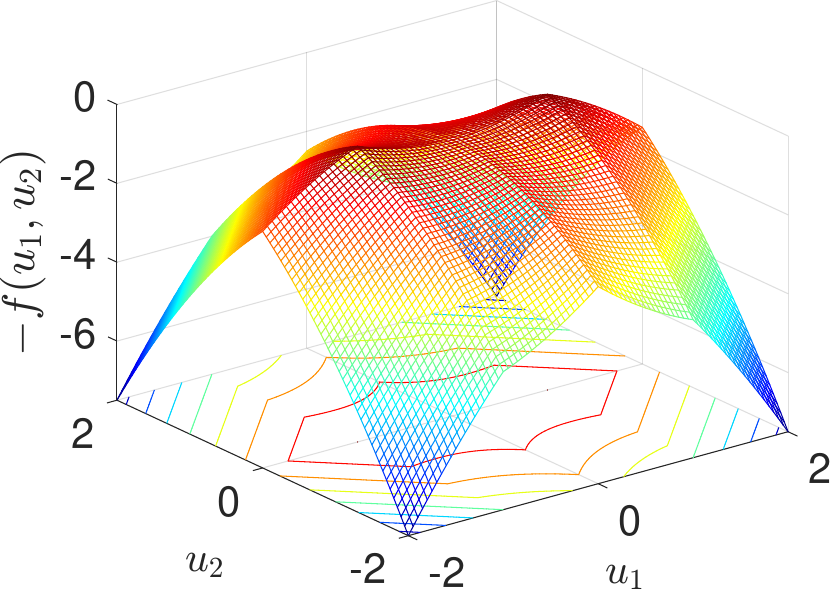}
    \caption{The function plot of $-f(\bu)$.}
    \label{fig:positive-curvature}
\end{wrapfigure}
Besides the above observation, at this place, we would like to exhibit another interesting distinction between $f$ and its $\ell_2$-norm counterpart: It is well understood that the Hessian matrix of the $\ell_2$-norm counterpart of $f$ is negative semidefinite at the origin (see, e.g.,~\cite[Theorem~2]{chi2019nonconvex},~\cite[Theorem~3]{li2019symmetry}, and also~\cite[Section~3.3]{li2019non}), but as the following example shows, this fact is no longer true in the $\ell_1$-norm case. (For a better and intuitive understanding of this phenomenon, we have visualized the function considered in the following example in the right panel. We remark that this visualization can also be used to verify the correctness of Theorem~\ref{thm:char-critical-r1} in this special case.)

\begin{example}\label{ex:positive-curvature}
    Suppose that $\bu^{\star}=(1,0)^{\T}$ and $\bw=(0,1)^{\T}$. Then, 
    by the reasoning and computations to be carried out in Theorem~\ref{thm:main-rank-one-outlier-free} later, we know
    $$
        d^2{f(\bd{0};\bd{0})}(\bw)=\max\left\{q_{22}:\bQ\in\R^{2\times 2},\,q_{11}=-1,\,q_{12},q_{22}\in[-1,1],\,q_{21}=q_{12}\right\}=1>0.
    $$
\end{example}

\subsection{The local functional behavior around the ground-truths}
With the set of stationary points of $f$ well understood, our next step is to analyze the local behavior of $f$ around each stationary point, with the start being the ground-truths. Before we proceed, we first remark that as $f$ is strictly continuous and subdifferentially regular everywhere, we know from~\cite[Theorem~9.16]{rockafellar2009variational} that $f$ is directionally differentiable, and we are thus legitimate to discuss its directional derivative. Due to symmetry, we only consider one of the ground-truths, namely $\bu^{\star}$.

\begin{proposition}\label{prop:rank-one-l1-sharpness}
    Suppose that $\bu^\star\in\R^n\setminus\{\bd{0}\}$. Then, it holds for any $\bw\in\R^n$ that
    $$
        d{f(\bu^{\star})}(\bw)\ge\min\left\{\alpha(\bu^{\star}),\,\frac{1}{2}\alpha(\bu^{\star})\cdot|\operatorname{supp}(\bu^{\star})|\right\}\cdot\|\bw\|_1,
    $$
    where $\alpha(\bu^{\star}):=\min\{|\bu_i^{\star}|:i\in\operatorname{supp}(\bu^{\star})\}>0$.
\end{proposition}

As a direct but powerful consequence of Proposition~\ref{prop:rank-one-l1-sharpness}, we observe through~\cite[Lemma~3.24]{bonnans2013perturbation} that the following first-order growth condition holds for $f$ at $\bu^{\star}$: There exists some universal constant $\beta>0$ (which depends on $\bu^{\star}$ only) such that
\begin{equation}\label{eq:l1-growth}
f(\bu)\ge f(\bu^{\star})+\beta\cdot\|\bu-\bu^{\star}\|_1,\quad\text{for every $\bu$ around $\bu^{\star}$}.
\end{equation}
This shows a very rapid growth rate (and in particular implies $\bu^{\star}$ is the minimum) therearound. Besides, (\ref{eq:l1-growth}) can be viewed as a stronger local version of the sharpness condition used in existing literature such as~\cite{davis2018subgradient,li2020nonconvex,davis2023stochastic} that is crucial for the local convergence rate of the subgradient method.
From another perspective, (\ref{eq:l1-growth}) can also be viewed as a stronger version of the well-known quadratic growth condition, which has equivalences to many interesting functional properties~\cite[Section~2]{rebjock2023fast} such as the Morse-Bott property~\cite[Definition~1]{rebjock2023fast}, the Polyak-Łojasiewicz inequality~\cite{lojasiewicz1963propriete,polyak1963gradient} (from which the Kurdyka-Łojasiewicz inequality~\cite{kurdyka1998gradients,attouch2010proximal} having even more equivalences to other functional properties~\cite[Theorem~18]{bolte2010characterizations} generalizes), and the error bound property~\cite{luo1993error}, and has led to a variety of remarkable algorithmic consequences (see, e.g.,~\cite{zhou2017unified,yue2019family,yue2019quadratic,li2022nonsmooth,zheng2023universal}). Owing to the above facts,
we believe Proposition~\ref{prop:rank-one-l1-sharpness} should be interesting in its own rights. 

\begin{proof}
    Because for $\bu=\bu^{\star}$ we have $\Signuuuu=[-1,1]^{n\times n}$, it follows from (\ref{eq:subdiff}) that
    $$
        \partial f(\bu^{\star})=\{\bZ\bu^{\star}:\bZ\in[-1,1]^{n\times n}\cap\R^{n\times n}_{\operatorname{sym}}\}.
    $$
    As a result, we know from~\cite[Theorem~9.16]{rockafellar2009variational} that
    $$
        d{f(\bu^{\star})}(\bw)=\max\left\{\langle\bZ,\bu^{\star}\bw^{\T}\rangle:\bZ\in[-1,1]^{n\times n}\cap\R^{n\times n}_{\operatorname{sym}}\right\},\quad\text{for any $\bw\in\R^n$},
    $$
    which is a linear program. This motivates us to study the following general linear program
    $$
        \max\left\{\langle\bZ,\bY\rangle:\bZ\in[-1,1]^{n\times n}\cap\R^{n\times n}_{\operatorname{sym}}\right\},
    $$
    where $\bY\in\R^{n\times n}$ is arbitrary. To begin with, we expand the above problem as
    $$
        \max\left\{\sum_{i=1}^n z_{i i}y_{i i}+\sum_{j>i} \left(z_{i j}y_{i j}+z_{j i}y_{j i}\right):\bZ\in[-1,1]^{n\times n}\cap\R^{n\times n}_{\operatorname{sym}}\right\}.
    $$
    Then, we analyze different subproblems respectively since they are independent of each other. It is easy to see that the optimal objective of the subproblem $\max\{z_{i i}y_{i i}: z_{i i}\in[-1,1]\}$ is $|y_{i i}|$. On the other hand, the other subproblem has the following equivalent reformulation
    $$
        \max\{z_{i j}y_{i j}+z_{j i}y_{j i}: z_{i j}=z_{j i},\,z_{i j},z_{j i}\in[-1,1]\}=\max\{z_{i j}(y_{i j}+y_{j i}): \,z_{i j}\in[-1,1]\},
    $$
    and thus the optimal objective is clearly $|y_{i j}+y_{j i}|$. Summarizing the above discussions shows
    $$
        \max\left\{\langle\bZ,\bY\rangle:\bZ\in[-1,1]^{n\times n}\cap\R^{n\times n}_{\operatorname{sym}}\right\}=\sum_{i=1}^n|y_{i i}|+\frac{1}{2}\sum_{j\neq i}|y_{i j}+y_{j i}|,
    $$
    which in particular implies
    $$
    \begin{aligned}
        d{f(\bu^{\star})}(\bw)&=\sum_{i=1}^n|u_i^{\star}w_i|+\frac{1}{2}\sum_{j\neq i}|u_i^{\star}w_j+u_j^{\star}w_i|
        \\&\ge\alpha(\bu^{\star})\cdot\left(\sum_{i\in\operatorname{supp}(\bu^{\star})}|w_i|\right)+\frac{1}{2}\sum_{i\in\operatorname{supp}(\bu^{\star})}\sum_{j\notin\operatorname{supp}(\bu^{\star})}|u_i^{\star}w_j+u_j^{\star}w_i|
        \\&\ge\alpha(\bu^{\star})\cdot\left(\sum_{i\in\operatorname{supp}(\bu^{\star})}|w_i|\right)+\frac{1}{2}\alpha(\bu^{\star})\cdot|\operatorname{supp}(\bu^{\star})|\cdot\left(\sum_{j\notin\operatorname{supp}(\bu^{\star})}|w_j|\right)
        \\&\ge\min\left\{\alpha(\bu^{\star}),\,\frac{1}{2}\alpha(\bu^{\star})\cdot|\operatorname{supp}(\bu^{\star})|\right\}\cdot\|\bw\|_1,
    \end{aligned}
    $$
    as desired.
\end{proof}

\subsection{An explicit characterization of the critical cone at spurious stationarities}
Before we proceed any further, as preparations for the subsequent analysis, we need to first present a technical lemma, and then derive a closed-form expression for the critical cone of $f$ at any spurious stationary point $\bu$, namely $\{\bw:d{f(\bu)}(\bw)=0\}$, where the lemma will be useful.
\begin{lemma}\label{lma:technical-1}
    It holds for any $s\neq 0$ that
    \begin{equation}\label{eq:technical-1}
        \left\{\begin{pmatrix}
            s\cdot p_1+q-r \\
            s\cdot p_2-q+r
        \end{pmatrix}:p_1,p_2,q,r\ge 0\right\}=\left\{\begin{pmatrix}
            x\\y
        \end{pmatrix}:s\cdot(x+y)\ge 0\right\}.
    \end{equation}
\end{lemma}

\begin{proof}
    Without loss of generality, we may assume that $s>0$. (The negative case can be handled in a symmetric way.) 
    The ``$\subseteq$'' direction should be clear, since we must have for any $p_1,p_2\ge 0$ that
    $$
        (s\cdot p_1+q-r)+(s\cdot p_2-q+r)=s\cdot(p_1+p_2)\ge 0.
    $$
    Conversely, suppose that we have some $(x,y)^{\T}$ with $x+y\ge 0$. Then, we manage to construct some $p_1,p_2,q,r\ge 0$ to represent $(x,y)^{\T}$ as in the left-hand-side set of (\ref{eq:technical-1}). Indeed, it is straightforward to check that the following parameterization 
    $$
        p_1=0\ge 0,\quad p_2=\frac{x+y}{s}\ge 0,\quad q=\frac{|x|+x}{2}\ge 0,\quad r=\frac{|x|-x}{2}\ge 0,
    $$
    suffices for the representation. This completes the proof.
\end{proof}

With Lemma~\ref{lma:technical-1} at hand, we next present the promised characterization of the critical cone.

\begin{proposition}\label{prop:critical-cone}
    Suppose that $\bu^\star\in\R^n\setminus\{\bd{0}\}$. Then, we have the following characterization of the critical cone of $f$ at any spurious stationary point $\bu$ that is nonzero
    $$
        \{\bw:d{f(\bu)}(\bw)=0\}=\prod_{j=1}^n
        \begin{dcases}
            \operatorname{Sign}(u_j)\cdot\R_-, & j\in\operatorname{supp}(\bu^{\star})\cap\bbJ_{=}(\bu), \\
            \R, & j\in\operatorname{supp}(\bu^{\star})\cap\bbJ_{<}(\bu), \\
            \{0\}, & j\notin\operatorname{supp}(\bu^{\star}),
        \end{dcases}
    $$
    and we simply have $\{\bw:d{f(\bd{0})}(\bw)=0\}=\R^n$.
\end{proposition}

\begin{proof}
    Similar to the proof of Theorem~\ref{thm:char-critical-r1}, we shall first be dealing with the problem under the assumption that $u^{\star}_i\neq 0$ for every $i\in[n]$; the remaining case will be handled afterwards. Due to the same reason as in the proof of Theorem~\ref{thm:char-critical-r1}, we further assume that $\bu$ and $\bu^{\star}$ have been simultaneously permuted such that $\bbJ_{=,\sim}(\bu)$, $\bbJ_{=,\centernot{\sim}}(\bu)$, and $\bbJ_{<}(\bu)$ form an \textit{ordered} partition of $[n]$. Let us begin with the most general case where $|\bbJ_{=,\sim}(\bu)|,|\bbJ_{=,\centernot{\sim}}(\bu)|\neq 0$ and revisit the boundary cases in the last; they will be rather clear after the general case has been fully understood. It can be easily observed that $\Signuuuu$ has the following structure in this case
    $$
        \begin{pmatrix}
            \bbT\subseteq[-1,1]^{|\bbJ_{=}(\bu)|\times|\bbJ_{=}(\bu)|} & \left(-\operatorname{Sign}(u_i^{\star} u_j^{\star}):i\in\bbJ_{=}(\bu),\,j\in\bbJ_{<}(\bu)\right) \\
            \left(-\operatorname{Sign}(u_i^{\star} u_j^{\star}):i\in\bbJ_{<}(\bu),\,j\in\bbJ_{=}(\bu)\right) & \left(-\operatorname{Sign}(u_i^{\star} u_j^{\star}):i\in\bbJ_{<}(\bu),\,j\in\bbJ_{<}(\bu)\right)
        \end{pmatrix},
    $$
    where $\bbT$ further admits the following structure
    $$
        \begin{pmatrix}
            [-1,1]^{|\bbJ_{=,\sim}(\bu)|\times|\bbJ_{=,\sim}(\bu)|} & \left(-\operatorname{Sign}(u_i^{\star} u_j^{\star}):i\in\bbJ_{=,\sim}(\bu),\,j\in\bbJ_{=,\centernot{\sim}}(\bu)\right) \\
            \left(-\operatorname{Sign}(u_i^{\star} u_j^{\star}):i\in\bbJ_{=,\centernot{\sim}}(\bu),\,j\in\bbJ_{=,\sim}(\bu)\right) & [-1,1]^{|\bbJ_{=,\centernot{\sim}}(\bu)|\times|\bbJ_{=,\centernot{\sim}}(\bu)|}
        \end{pmatrix}.
    $$
    With the above observation, it is routine to compute for any spurious stationary point $\bu$ that
    \begin{equation}\label{eq:rank-one-subdiff}
        \partial f(\bu)=\left\{\begin{pmatrix}
            \F_1(\bZ_1) \\ \F_2(\bZ_2) \\ \bd{0}
        \end{pmatrix}:
    \begin{aligned}
    &\bZ_1\in[-1,1]^{|\bbJ_{=,\sim}(\bu)|\times|\bbJ_{=,\sim}(\bu)|}\cap\R^{|\bbJ_{=,\sim}(\bu)|\times|\bbJ_{=,\sim}(\bu)|}_{\operatorname{sym}}
        \\&\bZ_2\in[-1,1]^{|\bbJ_{=,\centernot{\sim}}(\bu)|\times|\bbJ_{=,\centernot{\sim}}(\bu)|}\cap\R^{|\bbJ_{=,\centernot{\sim}}(\bu)|\times|\bbJ_{=,\centernot{\sim}}(\bu)|}_{\operatorname{sym}}
    \end{aligned}
    \right\},
    \end{equation}
    where we have used Theorem~\ref{thm:char-critical-r1} in computing each of the three blocks, and
    $$
        \F_1(\bZ_1):=\bZ_1\bu^{\star}_{\bbJ_{=,\sim}(\bu)}+\left\|\bu^{\star}_{\bbJ_{=,\sim}(\bu)}\right\|_1\cdot\operatorname{Sign}\left(\bu^{\star}_{\bbJ_{=,\sim}(\bu)}\right),
    $$
    $$
        \F_2(\bZ_2):=-\bZ_2\bu^{\star}_{\bbJ_{=,\centernot{\sim}}(\bu)}-\left\|\bu^{\star}_{\bbJ_{=,\centernot{\sim}}(\bu)}\right\|_1\cdot\operatorname{Sign}\left(\bu^{\star}_{\bbJ_{=,\centernot{\sim}}(\bu)}\right).
    $$
    Due to the fact that $\{\bw:d{f(\bu)}(\bw)=0\}=\N_{\partial f(\bu)}(\bd{0})$, which follows from the very definition, we next manage to characterize the normal cone of $\partial f(\bu)$ at the origin. It should be noted that by the representation (\ref{eq:rank-one-subdiff}), $\partial f(\bu)$ is the Cartesian product of three polyhedrons (where we have used the fact that the image of a polytope under a linear transformation is polyhedral~\cite[Theorem~19.3]{rock1997convex}). 
    As a result, by the product rule for computing normal cones~\cite[Proposition~6.41]{rockafellar2009variational} and the closedness of polyhedrons~\cite[Theorem~19.1]{rock1997convex}, it suffices to consider the three components in (\ref{eq:rank-one-subdiff}) respectively. Since $\N_{\{\bd{0}\}}(\bd{0})=\R^{|\bbJ_{<}(\bu)|}$ by convexity~\cite[Theorem~6.9]{rockafellar2009variational}, it remains to analyze the first two components. Careful readers may have already identified that they only differ a sign, and therefore let us abstract them out and analyze the normal cone of the following set at the origin
    $$
        \bbH(\bv):=\left\{\bZ\bv+\|\bv\|_1\cdot\operatorname{Sign}(\bv):\bZ\in[-1,1]^{n\times n}\cap\R^{n\times n}_{\operatorname{sym}}\right\},
    $$
    where $\bv\in\R^n$ is arbitrary but $\operatorname{supp}(\bv)=[n]$. Let us also make the following definition
    $$
        \bbH^{\prime}(\bv):=\left\{\bZ\bv:\bZ\in[-1,1]^{n\times n}\cap\R^{n\times n}_{\operatorname{sym}}\right\},
    $$
    which removes the intercept in $\bbH(\bv)$ for simplicity. Then, because $\bZ\bv=(\bI\boxtimes\bv^{\T})\operatorname{vec}(\bZ)$, we have by~\cite[Theorem~6.43]{rockafellar2009variational} that
    $$
        \N_{\bbH(\bv)}(\bd{0})=\N_{\bbH^{\prime}(\bv)}(-\|\bv\|_1\cdot\operatorname{Sign}(\bv))=\left\{\by:(\bI\boxtimes\bv)\by\in\N_{\operatorname{vec}([-1,1]^{n\times n}\cap\R^{n\times n}_{\operatorname{sym}})}\left(\operatorname{vec}(-\operatorname{Sign}(\bv\bv^{\T}))\right)\right\},
    $$
    where we have used the following facts that are easy to verify
    $$
        -\operatorname{Sign}(\bv\bv^{\T})\cdot\bv=-\|\bv\|_1\cdot\operatorname{Sign}(\bv)\quad\text{and}\quad -\operatorname{Sign}(\bv\bv^{\T})\in[-1,1]^{n\times n}\cap\R^{n\times n}_{\operatorname{sym}}.
    $$
    Because $\R^{n\times n}$ is isometrically isomorphic to $\R^{n^2}$, we next turn to study the representation for
    $$
        \N_{[-1,1]^{n\times n}\cap\R^{n\times n}_{\operatorname{sym}}}\left(-\operatorname{Sign}(\bv\bv^{\T})\right)\subseteq\R^{n\times n}.
    $$
    Notice that we can have a more explicit representation for the set
    $$
        [-1,1]^{n\times n}\cap\R^{n\times n}_{\operatorname{sym}}=\left\{\bZ:
        \begin{gathered}
            \langle\be_i\be_j^{\T},\bZ\rangle \le 1,\,\langle-\be_i\be_j^{\T},\bZ\rangle \le 1,\,\forall\,i,j\in[n] \\
            \langle\be_i\be_j^{\T},\bZ\rangle-\langle\be_j\be_i^{\T},\bZ\rangle\le 0,\,\langle\be_j\be_i^{\T},\bZ\rangle-\langle\be_i\be_j^{\T},\bZ\rangle\le 0,\,\forall\,j>i
        \end{gathered}
        \right\}.
    $$
    This, together with~\cite[Theorem~6.46]{rockafellar2009variational}, further implies $\N_{[-1,1]^{n\times n}\cap\R^{n\times n}_{\operatorname{sym}}}\big(-\operatorname{Sign}(\bv\bv^{\T})\big)$ equals
    $$
        \bbT:=\left\{-\sum_{i,j=1}^n p_{i j}\operatorname{Sign}(v_i v_j)\be_i\be_j^{\T}+\sum_{j>i}\left((q_{i j}-r_{i j})\cdot(\be_i\be_j^{\T}-\be_j\be_i^{\T})\right):p_{i j},q_{i j},r_{i j}\ge 0,\,\forall\, i,j\in[n]\right\}.
    $$
    We next give a more explicit representation for $\bbT$. It is easy to observe that $\bbT$ also admits some product structure. Indeed, the diagonal components of $\bbT$ are clearly independent of each other, and for each $i\in[n]$, we have
    $$
        \{t_{i i}:\bT\in\bbT\}=\{-p_{i i}\operatorname{Sign}(v_i^2):p_{i i}\ge 0\}=\R_{-}.
    $$
    On the other hand, it is clear that the off-diagonal component \textit{pairs} of $\bbT$ are independent of each other as well, and for each $j>i$, we have
    $$
    \begin{aligned}
        \left\{\begin{pmatrix}
            t_{i j} \\ t_{j i}
        \end{pmatrix}:\bT\in\bbT\right\}&=\left\{\begin{pmatrix}
            (-\operatorname{Sign}(v_i v_j))\cdot p_{i j}+q_{i j}-r_{i j} \\
            (-\operatorname{Sign}(v_j v_i))\cdot p_{j i}-q_{i j}+r_{i j}
        \end{pmatrix}:p_{i j},p_{j i},q_{i j},r_{i j}\ge 0\right\}
        \\&=\left\{\begin{pmatrix}
            x \\ y
        \end{pmatrix}:\operatorname{Sign}(v_i v_j)\cdot(x+y)\le 0\right\},
    \end{aligned}
    $$
    where in the last equality we have used Lemma~\ref{lma:technical-1}. Combining the above discussions leads to the following characterization
    $$
        \N_{[-1,1]^{n\times n}\cap\R^{n\times n}_{\operatorname{sym}}}\left(-\operatorname{Sign}(\bv\bv^{\T})\right)=\left\{\bT\in\R^{n\times n}:t_{i i}\le 0,\,\forall\,i\in[n],\,\operatorname{Sign}(v_i v_j)\cdot(t_{i j}+t_{j i})\le 0,\,\forall\,j>i\right\}.
    $$
    The above representation, together the fact that $(\bI\boxtimes\bv)\by=\operatorname{vec}(\bv\by^{\T})$, further implies
    $$
    \begin{aligned}
        \N_{\bbH(\bv)}(\bd{0})&=\left\{\by:\bv\by^{\T}\in\N_{[-1,1]^{n\times n}\cap\R^{n\times n}_{\operatorname{sym}}}\left(-\operatorname{Sign}(\bv\bv^{\T})\right)\right\}
        \\&=\left\{\by:v_i y_i\le 0,\,\forall\,i\in[n],\,\operatorname{Sign}(v_i v_j)\cdot(v_i y_j+v_j y_i)\le 0,\,\forall\,j>i\right\}
        \\&=\left\{\by:\operatorname{Sign}(v_i)\cdot y_i\le 0,\,\forall\,i\in[n],\,|v_i|\cdot\operatorname{Sign}(v_j)\cdot y_j+|v_j|\cdot\operatorname{Sign}(v_i)\cdot y_i\le 0,\,\forall\,j>i\right\}
        \\&=\left\{\by:\operatorname{Sign}(v_i)\cdot y_i\le 0,\,\forall\,i\in[n]\right\}=\prod_{i=1}^n(\operatorname{Sign}(v_i)\cdot\R_-),
    \end{aligned}
    $$
    which also admits a nice product structure. This general result, together with the scaling rule (which is a special case of~\cite[Theorem~6.43]{rockafellar2009variational}) and 
    the product rule~\cite[Proposition~6.41]{rockafellar2009variational} for computing normal cones, further implies
    $$
        \N_{\partial f(\bu)}(\bd{0})=\left(\prod_{j\in\bbJ_{=,\sim}(\bu)}\operatorname{Sign}(u^{\star}_j)\cdot\R_-\right)\times\left(\prod_{j\in\bbJ_{=,\centernot{\sim}}(\bu)}\operatorname{Sign}(u^{\star}_j)\cdot\R_+\right)\times\R^{|\bbJ_{<}(\bu)|},
    $$
    as desired. Now it is time to revisit the boundary cases where either $\bbJ_{=}(\bu)=\bbJ_{=,\sim}(\bu)$ or $\bbJ_{=}(\bu)=\bbJ_{=,\centernot{\sim}}(\bu)$. Actually, suppose that $\bbJ_{=}(\bu)=\bbJ_{=,\sim}(\bu)$. Then, we know the $\F_2(\bZ_2)$ term in (\ref{eq:rank-one-subdiff}) vanishes but the $\F_1(\bZ_1)$ term survived there with the same pattern. By applying the general result again, we see the above characterization of $\N_{\partial f(\bu)}(\bd{0})$ remains true for this case since $\prod_{j\in\bbJ_{=,\centernot{\sim}}(\bu)}\operatorname{Sign}(u^{\star}_j)\cdot\R_+$ now becomes nothing as $\bbJ_{=,\centernot{\sim}}(\bu)=\varnothing$. A symmetric argument shows this applies to the other case as well. With these materials established, all it remains is to simultaneously permute back $\bu$, $\bu^{\star}$, and $\N_{\partial f(\bu)}(\bd{0})$, and discuss the relationships between $\operatorname{Sign}(u_i^{\star})$ and $\operatorname{Sign}(u_i)$ for different $i\in[n]$, which exactly delivers the stated result specialized to this particular setting, as desired.

    As promised, we finally discuss how to handle the problem without the assumption made at the very beginning. Suppose that $\bu^{\star}$ has some zero coordinates. Because we know from Theorem~\ref{thm:char-critical-r1} that $u^{\star}_i=0$ implies $u_i=0$, standard computation shows each zero coordinate of $\bu^{\star}$ will contribute an independent dummy coordinate with value $\{\langle\bz,\bu\rangle:\bz\in[-1,1]^n\}$ to $\partial f(\bu)$. By convexity~\cite[Theorem~6.9]{rockafellar2009variational}, it should be rather clear that
    $$
        \N_{\{\langle\bz,\bu\rangle:\bz\in[-1,1]^n\}}(0)=\begin{dcases}
            \R, & \bu=\bd{0}, \\
            \{0\}, & \bu\neq\bd{0},
        \end{dcases}
    $$
    which, combined with the closedness of the set in (\ref{eq:rank-one-subdiff}) and the product rule for computing normal cones~\cite[Proposition~6.41]{rockafellar2009variational}, further implies the computation of the whole critical cone can be simply reduced to inserting the normal cones of these dummy coordinates as computed above after the remaining coordinates of the critical cone have all been computed, the latter of which has already been settled in the previous discussions. This completes the whole proof.
\end{proof}

\subsection{Putting everything together}
With the above preparations, we are now at a right place to accomplish the ultimate goal of this section by exhibiting ``negative curvatures'' at each of the spurious stationary points of $f$. Very naturally, for any such point $\bu$ of $f$, one may expect the directions pointing from $\bu$ to the ground-truths, namely $\bw:=\pm\bu^{\star}-\bu$, to be good candidates to fulfill the desired property. But to perform second-order analysis therealong, we must ensure that these directions are indeed within the critical cone of $f$, as otherwise the second subderivative $d^2{f(\bu;\bd{0})}(\bw)$ would explode to infinity; see~\cite[Proposition~13.5]{rockafellar2009variational}. The following lemma, which is a direct consequence of Proposition~\ref{prop:critical-cone}, points out this is indeed the case, desirably.

\begin{lemma}\label{lma:rank-one-w-critical}
    It holds for any spurious stationarity $\bu$ of $f$ that $d{f(\bu)}(\bw)=0$, where $\bw=\pm\bu^{\star}-\bu$.
\end{lemma}

\begin{proof}
    By virtue of Proposition~\ref{prop:critical-cone} and the fact that $\pm u^{\star}_j-u_j=0$ for all $j\notin\operatorname{supp}(\bu^{\star})$, which again follows from Theorem~\ref{thm:char-critical-r1}, it suffices to verify $\pm u^{\star}_j-u_j\in\operatorname{Sign}(u_j)\cdot\R_-$ for every $j\in\operatorname{supp}(\bu^{\star})\cap\bbJ_{=}(\bu)$. (We remark that by stating the above argument we have implicitly assumed that $|\operatorname{supp}(\bu^{\star})\cap\bbJ_{=}(\bu)|=0$. But this is without loss of generality, as otherwise the nonzero part of the critical cone would become the whole space, which contains everything, and the result thus holds vacuously.) We only need to discuss the following two cases:
    \begin{itemize}
        \item Suppose that $j\in\operatorname{supp}(\bu^{\star})\cap\bbJ_{=,\sim}(\bu)$, then we know $u^{\star}_j-u_j=0$ and $-u^{\star}_j-u_j=-2 u_j$, both of which are within $\operatorname{Sign}(u_j)\cdot\R_-$.
        \item Suppose that $j\in\operatorname{supp}(\bu^{\star})\cap\bbJ_{=,\centernot{\sim}}(\bu)$, then we know $u^{\star}_j-u_j=-2 u_j$ and $-u^{\star}_j-u_j=0$, both of which are within $\operatorname{Sign}(u_j)\cdot\R_-$.
    \end{itemize}
    Summarizing the above two cases leads to the desired claim.
\end{proof}

Equipped with Lemma~\ref{lma:rank-one-w-critical}, we are now ready to present the final results in this section.

\begin{theorem}\label{thm:main-rank-one-outlier-free}
    Suppose that $\bu^\star\in\R^n\setminus\{\bd{0}\}$. Then, for any spurious stationary point $\bu$ of $f$, we have for $\bw=\pm\bu^{\star}-\bu$ that
    $$
        d^2{f(\bu;\bd{0})}(\bw)=-\|\bu^{\star}\|_1^2<0.
    $$
\end{theorem}

\begin{proof}
    Once again, we make the assumption that $u^{\star}_i\neq 0$ for every $i\in[n]$, whose reason will be explained after the proof under this assumption has been completed. By virtue of Lemma~\ref{lma:rank-one-w-critical},~\cite[Theorem~13.14]{rockafellar2009variational},~\cite[Example~8.26]{rockafellar2009variational} (or more directly~\cite[Example~13.16]{rockafellar2009variational}, but we will not use the representation of the critical cone there), and the equivalent max-of-smooth representation
    $$
        f(\bu)=\max\left\{\frac{1}{2}\langle\bu\bu^{\T}-\bu^{\star}{\bu^{\star}}^{\T},\bP\rangle:\bP\in\{\pm 1\}^{n\times n}\cap\R^{n\times n}_{\operatorname{sym}}\right\},
    $$
    all it remains is to bound
    $$
    \max\left\{\bw^{\T}\bQ\bw:\bQ\in\bbQ(\bu)\right\}
    $$
    away from zero from the above, where
    $$
        \bbQ(\bu):=\left\{\sum_{\bP\in\bbI(\bu)}y_{\bP}\cdot\bP:y_{\bP}\ge 0,\,\forall\,\bP\in\bbI(\bu),\,\sum_{\bP\in\bbI(\bu)}y_{\bP}=1,\,\left(\sum_{\bP\in\bbI(\bu)}y_{\bP}\cdot\bP\right)\bu=\bd{0}\right\},
    $$
    and
    $$
        \bbI(\bu):=\operatorname{ext}\left(\Signuuuu\cap\R^{n\times n}_{\operatorname{sym}}\right).
    $$
    Due to the Krein-Milman theorem~\cite[Theorem~3.23]{rudin1991functional}, the compactness and convexity of $\Signuuuu\cap\R^{n\times n}_{\operatorname{sym}}$, and the well-known fact that the convex hull of any compact set is compact~\cite[Corollary~2.4]{barvinok2002course}, we know
    $$
    \begin{aligned}
        \bbQ(\bu)&=\operatorname{conv}(\bbI(\bu))\cap\{\bQ\in\R^{n\times n}:\bQ\bu=\bd{0}\}
        \\&=\left(\Signuuuu\cap\R^{n\times n}_{\operatorname{sym}}\right)\cap\{\bQ\in\R^{n\times n}:\bQ\bu=\bd{0}\}.
    \end{aligned}
    $$
    However, due to the fact that $(\Signuuuu\cap\R^{n\times n}_{\operatorname{sym}})\cdot\bu=\partial f(\bu)$ and the representation of $\partial f(\bu)$ established in (\ref{eq:rank-one-subdiff}), we know the only possible $\bZ\in\Signuuuu\cap\R^{n\times n}_{\operatorname{sym}}$ with $\bu$ inside its kernel is $\bZ=-\operatorname{Sign}(\bu^{\star}{\bu^{\star}}^{\T})$. As a result, we have the following characterization
    $$
        \bbQ(\bu)=-\operatorname{Sign}(\bu^{\star}{\bu^{\star}}^{\T}).
    $$
    With this characterization at hand, we are ready to perform the final computation. We first consider the case where $\bw=\bu^{\star}-\bu$. 
    In this case, it holds that
    \begin{equation*}
    \begin{aligned}
        \max\{\bw^{\T}\bQ\bw:\bQ\in\bbQ(\bu)\}&=\langle-\operatorname{Sign}(\bu^{\star}{\bu^{\star}}^{\T}),(\bu^{\star}-\bu)\cdot(\bu^{\star}-\bu)^{\T}\rangle
        \\&=\langle-\operatorname{Sign}(\bu^{\star}{\bu^{\star}}^{\T}),\bu^{\star}{\bu^{\star}}^{\T}\rangle=-\|\bu^{\star}\|_1^2<0,
    \end{aligned}
    \end{equation*}
    as desired, where in the second equality we have used the fact that $(\operatorname{Sign}(\bu^{\star}))^{\T}\bu=0$ (cf.~Theorem~\ref{thm:char-critical-r1}) and the symmetry of $-\operatorname{Sign}(\bu^{\star}{\bu^{\star}}^{\T})$. With the above computation, the symmetric case where $\bw=-\bu^{\star}-\bu$ can be carried out almost identically, and is thus omitted. 
    
    We next revisit the assumption that $u^{\star}_i\neq 0$ for every $i\in[n]$. Indeed, since $\pm u^{\star}_j-u_j=0$ for every $j\notin\operatorname{supp}(\bu^{\star})$ (cf.~Theorem~\ref{thm:char-critical-r1}), it follows that the columns and rows of $\bbQ(\bu)$ corresponding to the zero coordinates of $\bu^{\star}$ totally have nothing to do with the computation of $\max\{\bw^{\T}\bQ\bw:\bQ\in\bbQ(\bu)\}$. As a result, we may simply ignore all of these coordinates, and then apply the above analysis on the remaining coordinates to conclude the desired result. This completes the proof.
\end{proof}

Although Theorem~\ref{thm:main-rank-one-outlier-free} only states for the specific second subderivative $d^2{f(\bu;\bd{0})}(\bw)$, actually, it is possible to translate Theorem~\ref{thm:main-rank-one-outlier-free} into statements about the other constructions of second subderivatives introduced in Section~\ref{sec:diff-theory}. This is detailed in the following corollary.

\begin{corollary}\label{cor:main-rank-one-outlier-free}
    Suppose that $\bu^\star\in\R^n\setminus\{\bd{0}\}$. Then, for any spurious stationary point $\bu$ of $f$, we have for $\bw=\pm\bu^{\star}-\bu$ that:
    \begin{itemize}
        \item 
        $d^2{f(\bu})(\bw)=-\|\bu^{\star}\|_1^2<0$.
        \item 
        $d^{2}{f(\bu)}(\bw;\bz)\approx-\|\bu^{\star}\|_1^2<0$, for some $\bz\in\R^n$.
        \item 
        $f^{\prime\prime}(\bu;\bw)=-\|\bu^{\star}\|_1^2<0$.
    \end{itemize}
\end{corollary}

\begin{proof}
We prove the corollary in a point-to-point manner:
\begin{itemize}
    \item We first deal with the first point. From~\cite[Theorem~9.7]{rockafellar2009variational} and~\cite[Exercise~9.8(c)]{rockafellar2009variational}, we know $f$ is strictly continuous everywhere. This, combined with~\cite[Exercise~9.15]{rockafellar2009variational}, further implies $d{f(\bu)}$ is Lipschitz continuous. Hence, we know from~\cite[Definition~13.3]{rockafellar2009variational} and Lemma~\ref{lma:rank-one-w-critical} that
    $$
    d^2{f(\bu)}(\bw)=d^2{f(\bu;\bd{0})}(\bw),\quad\text{for $\bw=\pm\bu^{\star}-\bu$},
    $$
    and the desired result thus follows.
    
    \item We next proceed to the second point. Because $f$ is fully amenable as mentioned at the beginning of Section~\ref{sec:main-result}, we know from~\cite[Theorem~13.67]{rockafellar2009variational} that $f$ is also parabolically regular in the sense of~\cite[Definition~13.65]{rockafellar2009variational}. This, together with Lemma~\ref{lma:rank-one-w-critical} and the discussions between~\cite[Definition~13.65]{rockafellar2009variational} and~\cite[Theorem~13.66]{rockafellar2009variational}, further implies
    $$
    \inf\{d^2{f(\bu)}(\bw;\bz):\bz\in\R^n\}=d^2{f(\bu;\bd{0})}(\bw),\quad\text{for $\bw=\pm\bu^{\star}-\bu$},
    $$
    and the second result thus holds true, as desired. (We remark that as it is not known if the infimum above can indeed be attained, we have to use an approximation in the statement.)
    
    \item We finally work towards the last one. By~\cite[Proposition~8]{cui2020study}, we know $f$ is twice directionally differentiable, and $f^{\prime\prime}(\bu;\bw)$ thus exists for every $\bw\in\R^n$. This, together with Lemma~\ref{lma:rank-one-w-critical}
    and the definition of (one-sided) second directional derivative in~\cite[Equation~13.3]{rockafellar2009variational}, further implies
    $$
    f^{\prime\prime}(\bu;\bw)=d^2{f(\bu;\bd{0})}(\bw),\quad\text{for $\bw=\pm\bu^{\star}-\bu$},
    $$
    and hence the stated result, as desired.
\end{itemize}    
Summarizing the above three pieces completes the whole proof.
\end{proof}

With Corollary~\ref{cor:main-rank-one-outlier-free} at hand, we next provide some further explanations on the main finding.
\begin{remark}
    Although $f$ may not be twice semidifferentiable as per~\cite[Definition~13.6(a)]{rockafellar2009variational}, and we therefore usually do not have a second-order expansion of $f$ as in the smooth case (i.e., the Taylor expansion) as evidenced by~\cite[Exercise~13.7]{rockafellar2009variational}, we can still combine the last point in Corollary~\ref{cor:main-rank-one-outlier-free} with~\cite[Equation~13.3]{rockafellar2009variational} and Lemma~\ref{lma:rank-one-w-critical} to see that, at any spurious stationary point $\bu$ of $f$, there exists some universal constants $\beta,\delta>0$, such that
    $$
        f(\bu+t\bw)\le f(\bu)-\beta\cdot\frac{t^2}{2},\quad\text{for every $t\in(0,\delta)$}.
    $$
    In plain words, locally along the direction of $\bw$, the value of $f$ is upper dominated by a univariate strongly concave quadratic function with global maximum $f(\bu)$ achieved at $t=0$. This gives rise the possibility to further decrease the function value of $f$ at $\bu$ along $\bw$ at least as fast as decreasing some strongly concave quadratic function.
\end{remark}

By combining Theorem~\ref{thm:main-rank-one-outlier-free} with the second-order optimality condition in~\cite[Theorem~13.24(a)]{rockafellar2009variational}, 
we are finally led to the main conclusion of this section. (We remark that in the following corollary we have implicitly subsumed the case where $\bu^{\star}=\bd{0}$, but it should be rather clear that the statement is still valid for this trivial case.)
\begin{corollary}\label{cor:main}
    All second-order stationary points 
    of $f$ are globally optimal.
\end{corollary}

As a direct implication of Corollary~\ref{cor:main}, any algorithm for optimizing nonsmooth functions with convergence to second-order stationarity 
will avoid all spurious stationary points of $f$ and converge to its global optimality. Unfortunately, we are not aware of any algorithm with such a guarantee deterministically, even for fully amenable functions. (We remark that for smooth functions the existence of such algorithms is well known, e.g., the cubic Newton method~\cite{nesterov2006cubic}.) Nevertheless, for weakly-convex functions, there do exist some well-known results playing a similar role as~\cite{lee2016gradient}, which shows that for smooth functions the gradient descent method almost surely converges to local minimizers. We will revisit these results with negative observations in the forthcoming section. As another side remark, from a high-level algorithmic perspective, Corollary~\ref{cor:main} also implies that the study of any algorithm for the $\ell_1$-norm rank-one symmetric matrix factorization problem without convergence to its global optimality is actually meaningless, since it has been recently understood that for all this type of problems,
computing a stationary point is trivial~\cite[Remark~3.5]{guan2024subdiff}.

\clearpage

\section{Revisiting prior arts on nonsmooth global optimization}\label{sec:revisit}
With Theorem~\ref{thm:main-rank-one-outlier-free}, Corollary~\ref{cor:main-rank-one-outlier-free}, and Corollary~\ref{cor:main} established, as a very natural advance, one may be tempted to apply existing results on the generic minimizing behavior of simple algorithms (e.g., the subgradient method) such as~\cite{davis2021active,davis2022proximal,bianchi2024stochastic} that play a similar role as~\cite{lee2016gradient} under a nonsmooth setting to our problem. 
Unfortunately, we will show in the subsequent subsections that this may not be as successful as expected.

\subsection{The absence of active manifolds}
All of the aforementioned convergence results require at least the existence of a $C^2$-active manifold at each of the spurious stationary points of $f$;
see, e.g.,~\cite[Definitions~2.6-2.7]{davis2022proximal}. However, the following result shows that such an object can be absent even for a very simple instance of $f$.

\begin{proposition}\label{prop:absense-active-manifold}
    Suppose that $\bu^{\star}=(1,1)^{\T}$. Then, there does not exist a $C^p$-active manifold around $\bu_0:=(-1,1)^{\T}$, which is a spurious stationary point of $f$, of any order $p\ge 1$.
\end{proposition}

\begin{proof}
    For any set $\bbM\subseteq\R^2$ containing $\bu_0$ and $\varepsilon>0$ such that $\bbM\cap(\bu_0+\varepsilon\operatorname{int}(\bbB))$ is a $C^p$-smooth manifold for some $p\ge 1$ (as defined in~\cite[Definition~2.2]{davis2022proximal}) with the sharpness condition in~\cite[Definition~2.6]{davis2022proximal} fulfilled, we must have
    \begin{equation}\label{eq:active-manifold-temp1}
    (\bu_0+\varepsilon\operatorname{int}(\bbB))\cap\{(x,y)^{\T}:x+y=0\}\cap\{(x,y)^{\T}:|x|,|y|\le 1\}\subseteq\bbM\cap(\bu_0+\varepsilon\operatorname{int}(\bbB)),
    \end{equation}
    as otherwise there will be some point $\bu\in(\bu_0+\varepsilon\operatorname{int}(\bbB))\setminus\bbM$ with $\bd{0}\in\partial f(\bu)$ due to Theorem~\ref{thm:char-critical-r1}, which is a contradiction to the sharpness condition. We remark that because of (\ref{eq:active-manifold-temp1}), it should be rather clear that $\operatorname{dim}(\bbM\cap(\bu_0+\varepsilon\operatorname{int}(\bbB)))>0$.
    As $\bbM\cap(\bu_0+\varepsilon\operatorname{int}(\bbB))$ is more precisely a $C^p$-smooth submanifold embedded in $\R^2$ by~\cite[Definition~2.2]{davis2022proximal}, we know from\footnote{Although this theorem only states for $C^{\infty}$-smooth manifolds, it can be easily adapted to the $C^p$-smooth case.}~\cite[Theorem~3.12]{boumal2023introduction} that there exists another $\varepsilon^{\prime}>0$, an open subset $\bbV\subseteq\R^2$, and a $C^p$-diffeomorphism $\F:(\bu_0+\varepsilon^{\prime}\operatorname{int}(\bbB))\rightarrow\bbV$ such that
    \begin{equation}\label{eq:active-manifold-temp2}
        \F(\bbM\cap(\bu_0+\min\{\varepsilon,\varepsilon^{\prime}\}\operatorname{int}(\bbB)))=\bbE\cap\bbV=:\bbV^{\prime},
    \end{equation}
    where
    $$
        \bbE:=\R^{\operatorname{dim}(\bbM\cap(\bu_0+\varepsilon\operatorname{int}(\bbB)))}\times\{0\}^{2-\operatorname{dim}(\bbM\cap(\bu_0+\varepsilon\operatorname{int}(\bbB)))}
    $$
    is a subspace of $\R^2$. By (\ref{eq:active-manifold-temp1}), we know there must exist some $\delta>0$ such that the line segment
    $$
        \{(x,y)^{\T}:x+y=0\}\cap\{(x,y)^{\T}:1-\delta\le |x|,|y|\le 1\}\subseteq\bbM\cap(\bu_0+\min\{\varepsilon,\varepsilon^{\prime}\}\operatorname{int}(\bbB)).
    $$
    Next, we simply parameterize such a line segment through a curve
    $$
    \{c(t):t\in[0,1]\}\subseteq\bbM\cap(\bu_0+\min\{\varepsilon,\varepsilon^{\prime}\}\operatorname{int}(\bbB)),
    $$
    where $c:[0,1]\rightarrow\R^2$ with
    $$
        c(t):=\left(-(1-\delta+\delta t),1-\delta+\delta t\right)^{\T},
    $$
    and is thus $C^{\infty}$-smooth on $(0,1)$. By (\ref{eq:active-manifold-temp2}), we further know
    $$
    \{\F(c(t)):t\in[0,1]\}\subseteq\bbV^{\prime}.
    $$
    As a result, by the openness of $\bbV^{\prime}$ in the subspace topology induced by $\bbE$, we can further extend $\{\F(c(t)):t\in[0,1]\}$ towards both sides slightly while keeping the resulting curve remain in $\bbV^{\prime}$ and $C^1$-smooth. (This can be done by, e.g., slightly extending the curve at the two endpoints along the tangent directions.) We may override the above notation and simply use $\{\F(c(t)):t\in(-\rho,1+\rho)\}\subseteq\bbV^{\prime}$ to denote the extended $C^1$-smooth curve (with endpoints removed for simplicity), where $\rho>0$ is some small constant. By pulling such a curve back using $\F^{-1}$, which is $C^p$-smooth, and recall $p\ge 1$ and (\ref{eq:active-manifold-temp2}), we obtain another $C^1$-smooth curve (where we have abused our notation again)
    $$
    \{c(t):t\in(-\rho,1+\rho)\}\subseteq\bbM\cap(\bu_0+\min\{\varepsilon,\varepsilon^{\prime}\}\operatorname{int}(\bbB)),
    $$
    which is an extension of the original one. We next study the function behavior of $f$ when restricted to this curve at $t=1\in(-\rho,1+\rho)$. Actually, because the extended curve is $C^1$-smooth and thus $c^{\prime}(1)=(-\delta,\delta)^{\T}$, it is easy to see that there further exists some small constant $\kappa>0$ such that
    $$
        f(\bu)=\begin{dcases}
            \frac{1}{2}\left(4-u_1^2-u_2^2-2u_1 u_2\right), & \bu\in\{c(t):1-\kappa<t<1\}, \\
            \frac{1}{2}\left(u_1^2+u_2^2-2u_1 u_2\right), & \bu\in\{c(t):1<t<1+\kappa\},
        \end{dcases}
    $$
    each piece of which is clearly differentiable. This representation, together with the chain rule for differentiable functions and the continuous differentiability of $c$, further implies for $g(t):=f(c(t))$ that
    $$
        \lim_{t\nearrow 1} g^{\prime}(t)=0< 4\delta=\lim_{t\searrow 1}g^{\prime}(t),
    $$
    and thus $g$ can not even be continuously differentiable, i.e., $C^1$-smooth, showing the smoothness condition in~\cite[Definition~2.6]{davis2022proximal} can never be met even for $p=1$. This completes the proof.
\end{proof}

As a consequence, the aforementioned convergence results are all not applicable even to such a simple instance of $f$ that is well-structured and enjoys many benign properties.

\subsection{The potential risk of tilting a function}
As a remedy, one may also be tempted to, as a preprocessing procedure, linearly tilt the previous function slightly beforehand (i.e., superpose a small linear perturbation on the function), which can generically endow the function with the necessary properties~\cite[Theorem~2.9]{davis2022proximal}, and then apply the aforementioned results to the tilted function. However, in what follows, we will showcase the potential risk regarding such a procedure by exhibiting three examples. For a better and intuitive understanding of the insights delivered by these examples, we will accompany each of them a visualization after their discussions. The first example confirms the possibility of drifting the global minimizer of a function arbitrarily far away by tilting it, even if the perturbation is tiny.

\begin{example}\label{ex:tilt-1}
Consider the following function from $\R$ to $\R$, together with its gradient
    $$
        g(x):=\begin{dcases}
            0.5, & x\le-2, \\ 
            -\frac{(x+2)^{2}}{2}+0.5, & -2< x\le-1, \\
            \frac{x^{2}}{2}-0.5, & -1< x\le1, \\
            -\frac{(x-2)^{2}}{2}+0.5, & 1< x\le2, \\
            0.5, & x>2,
        \end{dcases}\qquad 
        g^{\prime}(x)=\begin{dcases}
            0, & x\le-2, \\ 
            -x-2, & -2< x\le-1, \\
            x, & -1< x\le1, \\
            -x+2, & 1< x\le2, \\
            0, & x>2.
        \end{dcases}
    $$
    It is clear that $\arg\min \{g(x):x\in\R\}=\{0\}$ and $g^{\prime}$ is Lipschitz continuous with modulus $1$. Thus, by~\cite[Proposition~4.12]{vial1983strong}, we know $g$ is $1$-weakly-convex. However, it is also easy to see that, for any $a\in\R\setminus\{0\}$, whose complement has a zero Lebesgue measure, we have
    $$
        \arg\min \{g(x)-a\cdot x:x\in\R\}=\{\operatorname{Sign}(a)\cdot\infty\},
    $$
    which arbitrarily deviates from the original minimizer.
\end{example}


\begin{figure}[!ht]
\centering
\subfloat[The function plot of $g(x)$.] {
\centering
        \resizebox{.45\linewidth}{!}{%
        \begin{tikzpicture}[
  declare function={
    func(\x)= (\x<=-2) * (0.5)   +
     and(\x>-2, \x<=-1) * (-0.5*(\x+2)*(\x+2)+0.5)     +
     and(\x>-1, \x<=1) * (0.5*\x*\x-0.5)     +
     and(\x>1,  \x<=2) * (-0.5*(\x-2)*(\x-2)+0.5) +
                (\x>2) * (0.5);
  }
]
\begin{axis}[
    axis lines = left,
    grid=major,
    xlabel = \(x\),
    ylabel = {\(g(x)\)},
height=0.5*\axisdefaultheight,
width=\axisdefaultheight,
]
\addplot [
    domain=-4:4, 
    samples=500,
    color=black,
]
{func(x)};

\end{axis}
\end{tikzpicture}}
}
\subfloat[The function plot of $g(x)+0.01\cdot x$.]{
\centering
        \resizebox{.45\linewidth}{!}{%
        \begin{tikzpicture}[
  declare function={
    func(\x)= (\x<=-2) * (0.5+0.01*\x)   +
     and(\x>-2, \x<=-1) * (-0.5*(\x+2)*(\x+2)+0.5+0.01*\x)     +
     and(\x>-1, \x<=1) * (0.5*\x*\x-0.5+0.01*\x)     +
     and(\x>1,  \x<=2) * (-0.5*(\x-2)*(\x-2)+0.5+0.01*\x) +
                (\x>2) * (0.5+0.01*\x);
  }
]
\begin{axis}[
    axis lines = left,
    grid=major,
    xlabel = \(x\),
    ylabel = {\(g(x)+0.01\cdot x\)},
height=0.5*\axisdefaultheight,
width=\axisdefaultheight,
]
\addplot [
    domain=-100:100, 
    samples=500,
    color=black,
]
{func(x)};

\end{axis}
\end{tikzpicture}}
}
\caption{The functions plots of $g(x)$ and $g(x)+0.01\cdot x$ for $g$ defined in Example~\ref{ex:tilt-1}.}
\label{fig:fcnplot-1}
\end{figure}
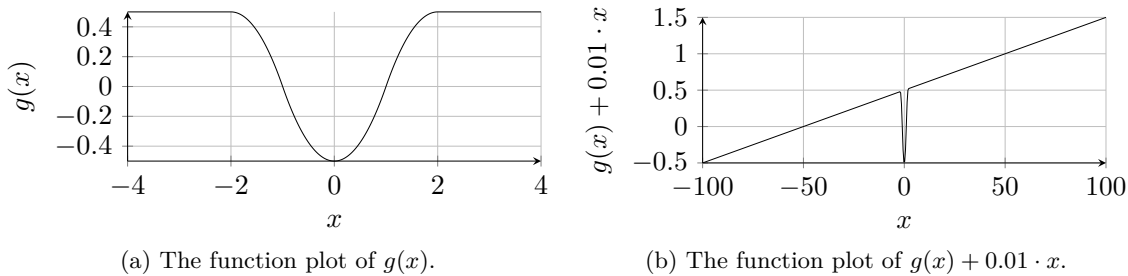

The next example witnesses the fact that introducing a countable number of additional local minima is also possible, even under a tiny perturbation.

\begin{example}\label{ex:tilt-2}
Consider the following function from $\R$ to $\R$ with its Fr\'echet subdifferential\footnote{We still use the notation $\partial g$ here to denote the Fr\'echet subdifferential as $g$ will soon turn out to be weakly-convex, and thus subdifferentially regular due to~\cite[Proposition~4.3]{vial1983strong},~\cite[Example~7.27]{rockafellar2009variational}, and~\cite[Exercise~8.20(b)]{rockafellar2009variational}.
}
    $$
        g(x):=\begin{dcases}
            -\frac{(x+\lfloor -x\rfloor+1)^{2}}{2}+\frac{1}{2}+\frac{\lfloor -x\rfloor}{2}, & x\le 0, \\
            -\frac{(x-\lfloor x\rfloor-1)^{2}}{2}+\frac{1}{2}+\frac{\lfloor x\rfloor}{2}, & x> 0,
        \end{dcases}\quad \partial g(x)=
        \begin{dcases}
            \{-(x-\lfloor x\rfloor)\}, & x<0,\,x\notin\bbZ, \\
            [-1,0], & x< 0,\,x\in\bbZ, \\
            [-1,1], & x= 0, \\
            \{-(x-\lfloor x\rfloor)+1\}, & x> 0,\,x\notin\bbZ, \\
            [0,1], & x> 0,\,x\in\bbZ,
        \end{dcases}
    $$
    the latter of which can be easily computed by combining the directional differentiability of $g$ (which is obvious from the definition) and~\cite[Exercise~8.4]{rockafellar2009variational}. It can be directly verified that
    $$
    (v-w)\cdot(x-y)\ge -(x-y)^2,\quad\text{for every $x,y\in\R$ and $(v,w)\in\partial g(x)\times \partial g(y)$},
    $$
    and thus $g$ is $1$-weakly-convex by~\cite[Lemma~2.1(4)]{davis2019stochastic}. Besides, it is also obvious that there is only one local (and thus global) minimizer of $g$, namely $x=0$. However, it actually holds, for $h_a(x):=g(x)-a\cdot x$ and any $a\in(-1,1)\setminus\{0\}$, which has a Lebesgue measure of $2$, that
    $$
        \partial h_a(x)=\partial g(x)-a=\operatorname{conv}(\{-a,\operatorname{Sign}(a)-a\}),
        \quad\text{for every $x$ with $\operatorname{Sign}(a)\cdot x >0$ and $x\in\bbZ$},
    $$
    where in the first equality we have used~\cite[Exercise~8.8(c)]{rockafellar2009variational}. Therefore, it holds for every $x$ with $\operatorname{Sign}(a)\cdot x >0$ and $x\in\bbZ$ that
    $$
        d{h_a(x)}(d)=\max\{-a\cdot d,(\operatorname{Sign}(a)-a)\cdot d\}\ge\min\{|a|,|\operatorname{Sign}(a)-a|\}\cdot |d|,\quad\text{for every $d\in\R$},
    $$
    where we have used~\cite[Theorem~8.30]{rockafellar2009variational} and the subdifferential regularity of $g$ as mentioned earlier. This, together with the positivity of $\min\{|a|,|\operatorname{Sign}(a)-a|\}$ and~\cite[Lemma~3.24]{bonnans2013perturbation}, further implies every $x$ with $\operatorname{Sign}(a)\cdot x >0$ and $x\in\bbZ$ is a local minimizer of $h_a$, all of which together has a cardinality of $\aleph_0$. We would like to remark that, since $\{x:0\in g(x)\}=\bbZ$, for any $a\in(-1,1)\setminus\{0\}$, any algorithm converging to any local minimizer of the perturbed function $h_a$ other than the origin is essentially stuck at a spurious stationary point of $g$.
\end{example}

Although in the above example the undesired perturbations only occupy a small region namely $(-1,1)\setminus\{0\}$,
we believe it still suffices to explain the potential risk of tilting, since in most scenarios introducing a large modification to the interested function is unfavorable. 

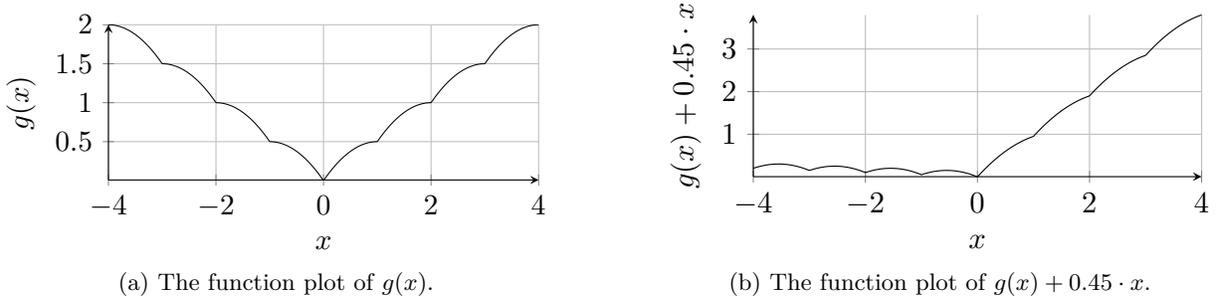
\begin{figure}[!ht]
\centering
\subfloat[The function plot of $g(x)$.] {
\centering
        \resizebox{.45\linewidth}{!}{%
        \begin{tikzpicture}[
  declare function={
    func(\x)= (\x<=0) * (-0.5*(\x+floor(-\x)+1)*(\x+floor(-\x)+1)+0.5+0.5*floor(-\x))   +
     (\x>0) * (-0.5*(\x-floor(\x)-1)*(\x-floor(\x)-1)+0.5+0.5*floor(\x));
  }
]
\begin{axis}[
    axis lines = left,
    grid=major,
    xlabel = \(x\),
    ylabel = {\(g(x)\)},
height=0.5*\axisdefaultheight,
width=\axisdefaultheight,
]
\addplot [
    domain=-4:4, 
    samples=500,
    color=black,
]
{func(x)};

\end{axis}
\end{tikzpicture}}
}
\hfill
\subfloat[The function plot of $g(x)+0.45\cdot x$.]{
\centering
        \resizebox{.45\linewidth}{!}{%
        \begin{tikzpicture}[
  declare function={
    func(\x)= (\x<=0) * (-0.5*(\x+floor(-\x)+1)*(\x+floor(-\x)+1)+0.5+0.5*floor(-\x)+0.45*\x)   +
     (\x>0) * (-0.5*(\x-floor(\x)-1)*(\x-floor(\x)-1)+0.5+0.5*floor(\x)+0.45*\x);
  }
]
\begin{axis}[
    axis lines = left,
    grid=major,
    xlabel = \(x\),
    ylabel = {\(g(x)+0.45\cdot x\)},
height=0.5*\axisdefaultheight,
width=\axisdefaultheight,
]
\addplot [
    domain=-4:4, 
    samples=500,
    color=black,
]
{func(x)};

\end{axis}
\end{tikzpicture}}
}
\caption{The functions plots of $g(x)$ and $g(x)+0.45\cdot x$ for $g$ defined in Example~\ref{ex:tilt-2}.}
\label{fig:fcnplot-2}
\end{figure}

The final example revisits the function designed and investigated in Proposition~\ref{prop:absense-active-manifold} and demonstrates that the undesirable phenomenon of transforming a spurious stationary point of the original function to a local minimizer of the tilted function can also take place in $\ell_1$-norm symmetric matrix factorization. Because its reasoning is similar to that of Example~\ref{ex:tilt-2}, the references corresponding to the underpinning variational analysis theories will be hidden in what follows.

\begin{example}\label{ex:tilt-3}
    Suppose that $\bu^{\star}=(1,1)^{\T}$ and let $\bu_0:=(-1,1)^{\T}$, which is a spurious stationary point of $f$ (cf.~Theorem~\ref{thm:char-critical-r1}). Through the very definition, it can be easily computed that
    $$
        \partial f(\bu_0)=\left(\begin{pmatrix}
            \operatorname{Sign}(0) & \operatorname{Sign}(-2) \\
            \operatorname{Sign}(-2) & \operatorname{Sign}(0)
        \end{pmatrix}\cap\R^{2\times 2}_{\operatorname{sym}}\right)\cdot \begin{pmatrix}
            -1 \\ 1
        \end{pmatrix}=\begin{pmatrix}
            [-1,1] & -1 \\
            -1 & [-1,1]
        \end{pmatrix}\cdot \begin{pmatrix}
            -1 \\ 1
        \end{pmatrix}=\begin{pmatrix}
            [-2,0] \\ 
            [0, 2]
        \end{pmatrix}.
    $$
    However, it actually holds, for $h_{\ba}(\bu):=f(\bu)-\ba^{\T}\bu$ and any $\ba\in(-2,0)\times(0,2)$, which has a Lebesgue measure of $4$, that
    $$
        \partial h_{\ba}(\bu_0)=\partial f(\bu_0)-\ba=[-2-a_1,-a_1]\times[-a_2,2-a_2].
    $$
    As a result, we have
    $$
    \begin{aligned}
        d{h_{\ba}(\bu_0)}(\bdd)&=\max\{(-2-a_1)\cdot d_1,-a_1\cdot d_1\}+\max\{-a_2\cdot d_2,(2-a_2)\cdot d_2\}
        \\&\ge\min\{|2+a_1|,|a_1|,|a_2|,|2-a_2|\}\cdot\|\bdd\|_1,\quad\text{for every $\bdd\in\R^2$},
    \end{aligned}
    $$
    which, together with the positivity of $\min\{|2+a_1|,|a_1|,|a_2|,|2-a_2|\}$, further implies $\bu_0$ is a local minimizer of $h_{\ba}$. Thus, similar to Example~\ref{ex:tilt-2}, any algorithm running on $h_{\ba}$ for $\ba\in(-2,0)\times(0,2)$ with convergence to its local minimizer $\bu_0$ is actually trapped by this spurious stationary point. By applying a completely symmetric argument, one can also easily discover the same dilemma at $-\bu_0$, which is a spurious stationary point of $f$ as well.
\end{example}


\begin{figure}[!ht]
\centering
\subfloat[The function plot of $-f(\bu)$.]{\includegraphics[width=0.45\linewidth]{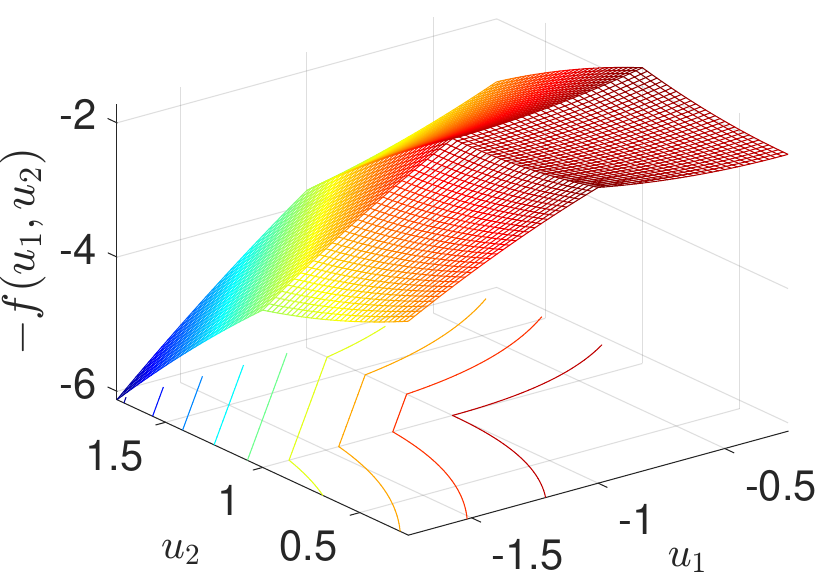}}
\hspace{0.05\linewidth}
\subfloat[The function plot of $-f(\bu)+\bu_0^{\T}\bu$.]{\includegraphics[width=0.45\linewidth]{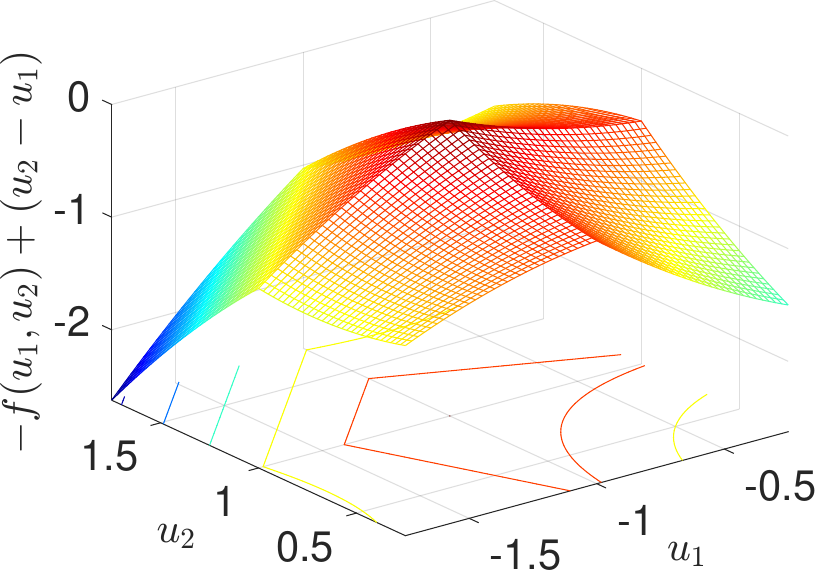}}
\caption{The functions plots of (negated) $f(\bu)$ and $f(\bu)-\bu_0^{\T}\bu$ around $\bu_0$ when $\bu^{\star}=(1,1)^{\T}$.}
\label{fig:fcnplot-3}
\end{figure}

In summary, the above examples deliver a message that tilting a function can potentially lead to some negative consequences. 

\subsection{Prospects}
\begin{wrapfigure}[14]{r}{0.35\linewidth}
\vspace{-1em}
    \centering
    \includegraphics[width=\linewidth]{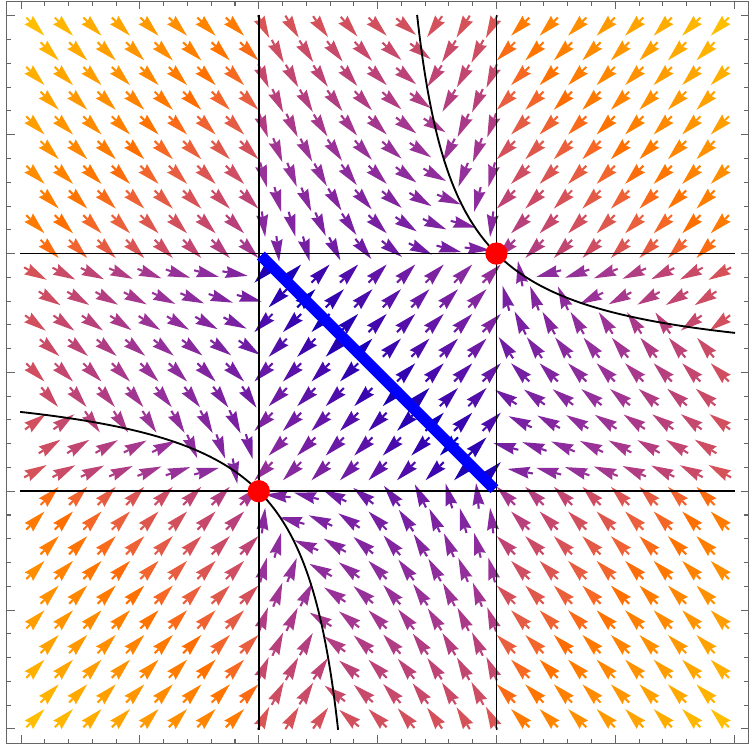}
    \caption{The (negative) subgradient flow on a specific instance of $f$.}
    \label{fig:subgradient-flow}
\end{wrapfigure}
Despite the above negative sides, we highlight that the situation on the global optimization of $f$ might be much more benign than our expectation. This is motivated by the following observation: We have visualized the (negative) subgradient flow on the specific instance of $f$ investigated in Proposition~\ref{prop:absense-active-manifold} and Example~\ref{ex:tilt-3}, which is shown in the right panel, where the thick line segment stands for the set of spurious stationary points and the two dots the ground-truths. The visualization clearly suggests that with generic initializations (or more specifically, without being initialized right on the one-dimensional subspace $\{(x,y)^{\T}:x+y=0\}$), the (negative) subgradient flow will eventually lead to the global optimality of this function. Owing to this observation, we would like to make the following conjecture.
\vspace{\baselineskip}

\begin{conjecture}\label{conj:convergence}
    For the $\ell_1$-norm rank-one symmetric matrix factorization problem, the subgradient method with random initialization (with distribution absolutely continuous w.r.t.\ the Lebesgue measure) and sufficiently small step sizes that are also diminishing (i.e., positive, not summable, and having a limit of zero) converges to its global optimality almost surely.
\end{conjecture}

\section{Concluding remarks and future works}\label{sec:conclusion}
In this paper, we have investigated the nonsmooth optimization landscape of the $\ell_1$-norm rank-one symmetric matrix factorization problem. Specifically, we have completely characterized the set of stationary points of the problem. Based on the characterization, we have shown that at each spurious stationary point of the problem, there always exists some direction, along which the second subderivatives are negative, thus proving all second-order stationary points are actually globally optimal. With the above developments, we have revisited existing results on the generic minimizing behavior of simple algorithms for nonsmooth optimization and showcased the potential risk of their applications to our problem.

This paper also opens up several interesting (but can be rather challenging) directions for future studies. As the most natural direction, since the main purpose of introducing the $\ell_1$-norm for rank-one symmetric matrix factorization is to enhance its robustness to outliers, investigating the outlier conditions under which the inexistence of spurious second-order stationary points still holds true is a necessity. Besides, since this paper only considers the rank-one symmetric case of the $\ell_1$-norm matrix factorization problem, it is also interesting and meaningful to study the optimization landscapes of its asymmetric and higher-rank counterparts via second-order variational analysis. 
Moreover, as we have already discussed, investigating the algorithm design for either deterministically or generically converging to some second-order stationary point of a nonsmooth function is also promising and significant. Finally, with the machinery developed in this paper, it becomes more possible to perform landscape analysis for a variety of other more sophisticated nonsmooth learning problems, such as robust phase retrieval~\cite{duchi2019solving,davis2020nonsmooth}, robust low-rank matrix recovery~\cite{li2020nonconvex,ma2023global}, neural collapse with robust losses (see, e.g.,~\cite{zhu2021geometric,zhou2022optimization,yaras2022neural,zhou2022all,li2024neural} for its smooth counterparts), robust principal component analysis~\cite{wang2023linear,zheng2022linearly} (see also the references therein), robust feature extraction~\cite{chen2023robust}, and robust rotation group synchronization~\cite{liu2024resync} (as well as the robust counterparts of many other synchronization problems such as~\cite{liu2017estimation,zhu2021orthogonal,liu2023unified}).

\section*{Acknowledgments} 
The authors would like express their sincere gratitude to \href{https://www1.se.cuhk.edu.hk/~tianlai/}{Lai Tian} at the Chinese University of Hong Kong for reminding them that tilting a function can be dangerous and many helpful discussions at the early stage of the project, as well as \href{https://wangjinxin-terry.github.io/}{Jinxin Wang} at the Chinese University of Hong Kong for his reading and comments, which in particular renew their attention to the paper~\cite{davis2020nonsmooth}.

\bibliographystyle{ieeetr}
\bibliography{references.bib}{}
\end{document}